\documentclass[oneside,english]{amsart}
\usepackage[T1]{fontenc}
\usepackage[latin9]{inputenc}
\usepackage{geometry}
\usepackage{amsbsy,verbatim}
\usepackage{amstext,txfonts}
\usepackage{amsthm}
\usepackage{amssymb}
\usepackage{scalerel,stackengine}
\stackMath
\newcommand\reallywidehat[1]{%
\savestack{\tmpbox}{\stretchto{%
  \scaleto{%
    \scalerel*[\widthof{\ensuremath{#1}}]{\kern-.6pt\bigwedge\kern-.6pt}%
    {\rule[-\textheight/2]{1ex}{\textheight}}%WIDTH-LIMITED BIG WEDGE
  }{\textheight}% 
}{0.5ex}}%
\stackon[1pt]{#1}{\tmpbox}%
}

\makeatletter
\numberwithin{equation}{section}
\numberwithin{figure}{section}
\theoremstyle{plain}
\newtheorem{thm}{\protect\theoremname}[section]
\theoremstyle{plain}
\newtheorem{cor}[thm]{\protect\corollaryname}
\theoremstyle{plain}
\newtheorem{lem}[thm]{\protect\lemmaname}
\newtheorem*{thm*}{\protect\theoremname}
\ifx\proof\undefined
\newenvironment{proof}[1][\protect\proofname]{\par
	\normalfont\topsep6\p@\@plus6\p@\relax
	\trivlist
	\itemindent\parindent
	\item[\hskip\labelsep\scshape #1]\ignorespaces
}{%
	\endtrivlist\@endpefalse
}
\providecommand{\proofname}{Proof}
\fi
\theoremstyle{plain}
\newtheorem{prop}[thm]{\protect\propositionname}
\theoremstyle{plain}

\theoremstyle{plain}
\newtheorem{claim}[thm]{\protect\claimname}
\theoremstyle{plain}

\theoremstyle{plain}
\newtheorem*{prop*}{\protect\propositionname}
\theoremstyle{definition}
\newtheorem{defn}[thm]{\protect\definitionname}
\theoremstyle{plain}
\newtheorem*{lem*}{\protect\lemmaname}
\theoremstyle{remark}

\usepackage{babel}

\makeatother
\newcommand{\gnote}[1]{}      %{{$^\rightarrow$\marginpar{\footnotesize {\bf g:} #1}}}

\usepackage{babel}
\providecommand{\conjname}{Conjecture}
\providecommand{\claimname}{Claim}
\providecommand{\corollaryname}{Corollary}
\providecommand{\definitionname}{Definition}
\providecommand{\lemmaname}{Lemma}
\providecommand{\propositionname}{Proposition}
\providecommand{\theoremname}{Theorem}
\providecommand{\hypothesisname}{Hypothesis}
\newcommand{\remove}[1]{}

\begin{document}
\global\long\def\connected{\text{highly connected}}%
\global\long\def\f{\mathcal{F}}%
\global\long\def\E{mathcal{E}}%
\global\long\def\a{\mathcal{A}}% 
\global\long\def\pn{\mathcal{P}\left(\left[n\right]\right)}%
\global\long\def\g{\mathcal{G}}%
\global\long\def\Hom{\mathrm{Hom}}% 
\global\long\def\l{\mathcal{L}}% 
\global\long\def\s{\mathcal{S}}%
\global\long\def\j{\mathcal{J}}%
\global\long\def\d{\mathcal{D}}%
\global\long\def\Cay{\mathrm{Cay}}%
\global\long\def\OPT{\mathrm{OPT}}
\global\long\def\Image{\mathrm{Im}}%
\global\long\def\supp{\mathrm{supp}} 
\global\long\def\GL{\mathrm{GL}}%
%\global\long\def\SL{\mathrm{SL}}% 
\global\long\def\Inf{}%
\global\long\def\Id{\textrm{Id}}%
\global\long\def\Tr{\mathrm{Tr}}%
\global\long\def\sgn{\textrm{sgn}}%
\global\long\def\p{\mathcal{P}}%
\global\long\def\h{\mathcal{H}}%
\global\long\def\n{\mathbb{N}}%
\global\long\def\a{\mathcal{A}}%
\global\long\def\b{\mB}%
\global\long\def\c{\mathcal{C}}%
\global\long\def\e{\bE}%
\global\long\def\x{\mathbf{x}}%
\global\long\def\y{\mathbf{y}}%
\global\long\def\z{\mathbf{z}}%
\global\long\def\c{\mathcal{C}}%
\global\long\def\av{\mathsf{A}}%
\global\long\def\chop{\mathrm{Chop}}%
\global\long\def\stab{\mathrm{Stab}}%
\global\long\def\Span{\mathrm{Span}}%
\global\long\def\Domain{\mathrm{Domain}}%
\global\long\def\codim{\mathrm{codim}}%
\global\long\def\Var{\mathrm{Var}}%
\global\long\def\rank{\mathrm{rank}}%
\global\long\def\t{\mathsf{T}}%

\newcommand\F[1]{\bF_#1}
\newcommand\Sl[2]{\mathrm{SL}_{#1}(\F{#2})}
\newcommand\Gl[2]{\mathrm{GL}_{#1}(\F{#2})}
\newcommand\SLnq{\Sl{n}{q}}
\newcommand\GLnq{\Gl{n}{q}}
\newcommand\SLV{\mathrm{SL}(V)}
\newcommand\GLV{\mathrm{GL}(V)}
\newcommand{\N}{\mathbb{N}}
\newcommand{\Z}{\mathbb{Z}}
\newcommand{\R}{\mathbb{R}}
\newcommand{\C}{\mathbb{C}}
\newcommand{\bE}{\mathbb{E}}
\newcommand{\bF}{\mathbb{F}}
\newcommand{\mE}{\mathcal{E}}
\newcommand{\mB}{\mathcal{B}}
\newcommand{\mT}{\mathcal{T}}
\newcommand{\im}{\mathrm{im}}

\newcommand{\set}[1]{\left\{ #1 \right\}}
\newcommand{\eqdef}{\stackrel{\text{def}}{=}}
\newcommand{\sbinom}[2]{\genfrac{[}{]}{0pt}{}{#1}{#2}}
 
\global\long\def\sqbinom#1#2{\left[\begin{array}{c} #1\\ #2 \end{array}\right]}%

\title{Polynomial Bogolyubov for special linear groups via tensor rank}
\author{Shai Evra, Guy Kindler, and Noam Lifshitz}

\begin{abstract}
We prove a polynomial Bogolyubov type lemma for the special linear group over finite fields. 
Specifically, we show that there exists an absolute constant $C>0,$ such that if $A$ is a density $\alpha$ subset of the special linear group, then the set  $AA^{-1}AA^{-1}$ contains a subgroup $H$ of density $\alpha^C$.
Moreover, this subgroup is isomorphic to a special linear group of a smaller rank. 
We also show that if $A$ is an approximate subgroups then it can be covered by the union of few cosets of $H$.
Our proof makes use of the Gurevich--Howe notion of tensor rank, and of a strengthened Bonami type Lemma for global functions on the bilinear scheme. 
We also present applications to spectral bounds for global convolution operators, global product free sets, and covering numbers corresponding to global sets.
\end{abstract}

\maketitle

\section{Introduction}
Bogolyubov's lemma for finite fields \cite{Bogo39} states that for a dense-enough set $A\subseteq \F{q}^n$, the set $2A-2A$ contains a large subspace. 
The state-of-the-art in this direction was proven by Sanders~\cite{sanders2012bogolyubov} who showed that if $A\subseteq \F{q}^n$ has density $\alpha$ then $2A-2A$ contains a subspace of co-dimension $O_q(\log^4(\frac{1}{\alpha}))$. 
This is refered to as a quasi-polynomial Bogolyubov result, as the density of the subspace is quasi-polynomial in the density of $A$. 
It is a major open problem in additive combinatorics to obtain a polynomial version of the Bogolyugov lemma. 

In this work we prove an analogue result in $\SLnq$, showing that for a  subset $A\subseteq \SLnq$, of density $\mu(A) = |A|/|\SLnq|$, the set $A A^{-1}A A^{-1}$ contains a subgroup $L$ whose density is polynomial in the density of $A$, thereby showing a polynomial Bogolyubov type result for $\SLnq$. 
Moreover, we show that $L$ can be taken to be of a certain `dictatorial' structure. 
Following Friedgut~\cite{friedgut2008measure} we call the set of matrices of the form $\d_{v,u} := \{g\in\SLnq : gv=u\}$ and of the form $\d^*_{v,u} := \{g\in \SLnq : g^t v = u \}$ \emph{dictators}. 
If $k$ dictators have a nonempty intersection, and their intersection is not the intersection of any $k-1$ dictators, then we call their intersection a $k$-\emph{umvirate}.
Our polynomial Bogolyubov lemma gaurantees that $AA^{-1}AA^{-1}$ contains a subgroup that is also an umvirate -- we call these \emph{groumvirates}. 
A particularly nice class of groumvirates are the following subgroups.

\begin{defn} \label{def:groumvirate}
A \emph{good $k$-groumvirate} in $\SLnq$ is a conjugate of the subgroup of matrices of the form
\[
L_k = \left\lbrace \begin{pmatrix} I_k & 0\\ 0 & X \end{pmatrix} \,:\, X\in \Sl{n-k}{q}  \right\rbrace,
\]
where $I_k$ is the $k\times k$ identity matrix.
We call a coset of a good $k$-groumvirate a good \emph{$k$-umvirate}.
\end{defn}

Our polynomial variant of the Bogolyubov lemma takes the following form. 

\begin{thm}\label{thm:Bogolyubov}
There exists $C>0$, such that for every $n\in \N$, every prime power $q$ and every $A \subseteq \SLnq$, the set $AA^{-1}AA^{-1}$ contains a good groumvirate of density at least $\mu(A)^C$.
\end{thm}

We prove Theorem~\ref{thm:Bogolyubov} by first finding a good $2k$-umvirate $A \subset U \subset \SLnq$, in which $A$ satisfies a certain pseudorandomness notion called globalness. 
We then prove that global sets have good growth properties by showing that if $A$ and $B$ are global sets, then $AB$ covers most of $\SLnq$. 
Hence for a global set $A$, the density of $AA^{-1}$ in $\SLnq$ is greater than $1/2$, and therefore its square $AA^{-1}AA^{-1}$ covers all of the group.  

\subsection{Global sets and mixing}
We actually prove a stronger statement that implies growth, namely we show that the convolution of the indicators of global sets (defined below) is very close to constant. 

Let $L^2(\SLnq) = \{f \colon \SLnq \to \C\}$, endowed with the \emph{convolution} operation defined for any $f,g \in L^2(\SLnq)$ by $f*g(x) = \bE_{y\sim \SLnq}[f(xy^{-1})g(y)]$, where we denote $y\sim \SLnq$ to mean that $y$ is chosen uniformly at random from $\SLnq$.
For any subset $A\subseteq \SLnq$ denote its indicator function by $1_A \colon \SLnq \to \{0,1\}$, and note that $\bE[1_A] = \mu(A)$, the density of $A$.

\begin{thm} \label{thm:SLn-mixing-global}
There exists $c>0$, such that for any $n \in \N$ and any prime power $q$, the following holds. 
Let $A,B\subseteq \SLnq$ be two global sets (see Definition~\ref{def:global} below) of density $\mu(A),\mu(B) \ge q^{-cn^2}$.
Then 
\[
\|\; 1_A * 1_B - \mu(A)\mu(B) \;\|_2 \le q^{-n/4} \mu(A)\mu(B).
\]
\end{thm}

In order to define globalness (as well as for other purposes) it is convenient to consider the set of invertible matrices as a subset of the abelian group of linear maps from $V$ to itself, where $V \cong \F{q}^n$.
More generally, for any two linear subspace $V$ and $W$ over $\F{q}$ we denote by $\l(V,W)$ the space of linear maps from $V$ to $W$. 
The set $\l(V,W)$ is also known as the \emph{bilinear scheme}. 
Note that for $V=\F{q}^n$ it holds that  $\SLnq \cong \SLV \subset \GLV \subset \l(V,V)$. 
The bilinear scheme is equipped with $i$-umvirates that are defined analagously to the definition for $\SLnq$. 
This allows us to talk about restrictions of functions that are defined over the bilinear scheme.

\begin{defn}[restrictions for functions of $\l(V,W)$]
For any pair of subspaces $V' \le V$ and $W' \le W$, we identify $\l(V/V',W')$ with the subspace of linear maps $T\in \l(V,W)$ such that $V' \le \ker T$ and $\im T \le W'$.
Given an operator $T\in \l(V,W)$, and a pair of subspaces $V' \le V$ and $W' \le W$, for any function $f \in L^2(\l(V,W))$, define its \emph{restriction}, w.r.t. $V'$, $W'$ and $T$, to be
\[
f_{(V',W')\to T} \in L^2(\l(V/V',W')), \qquad f_{(V',W')\to T}(S)= f(S+T).
\]
For $d\in \N$, a $d$-\emph{restriction} of $f$ is a restriction of the form $f_{(V',W')\to T}$, where $d= \dim V' + \mathrm{codim} W'$.

\end{defn}

The following notion of globalness for linear maps is due to Ellis, Kindler and Lifshitz~\cite{ellis2022analogue} (a somewhat analogue notion appeared in~\cite{dinur2018non, khot2018pseudorandom} in the context of functions over vector spaces).  

\begin{defn}[globalness for functions and subsets of $\l(V,W)$] \label{def:global}
A function $f \in L^2(\l(V,W))$ is said to be \emph{$(d,\epsilon)$-global} if for any $d$-restriction of it $f_{(V',W')\to T}$, we have
\[
\|f_{(V',W')\to T}\|_2^2 \le \epsilon.
\]
We also fix a small consant $\zeta>0$ once and for all and say that $f$ is \emph{global} if it is $(d,q^{\zeta d n}\|f\|_2^2)$-global for all $d$. 
We say that a nonempty set $A \subset \l(V,W)$ is \emph{global} if its indicator function $1_A$ is global. 
\end{defn}

\subsection{Product mixing}
Using similar methods to Theorem~\ref{thm:SLn-mixing-global}, we also prove a three-function version.
Let $\langle , \rangle$ be the standard inner product on $L^2(\SLnq)$.

\begin{thm}\label{thm:SLn-product-mixing} 
Let $A,B,C \subseteq \SLnq$ be global sets. 
Then
\[
\left|\; \langle 1_A*1_B,1_C \rangle - \mu(A) \mu(B) \mu(C) \; \right| \le q^{-n/5} \mu(A) \mu(B) \mu(C).
\] 
\end{thm}

Using an observation by Nikolov and Pyber~\cite{nikolov2011product}, this yields the following corrolary of Theorem~\ref{thm:SLn-product-mixing}.

\begin{cor}\label{cor:nikolov-pyber}
If $A,B,C \subseteq \SLnq$ are global sets, then $ABC = \SLnq$.
\end{cor}

To put our result in context we note that Gowers~\cite{gowers} proved an analogue of Theorem~\ref{thm:SLn-product-mixing} where the globalness hypothesis is replaced by the hypothesis that the sets $A,B,C$ all have density at least $\left(\frac{q^n-1}{q-1}\right)^{-1/3}$.
Gowers was motivated by the problem of finding the largest product free set in $\SLnq$, where $A$ is said to be \emph{product free} if $A^2\cap A=\varnothing$. 
He managed to prove that such sets must have density $\le \left( \frac{q^n-1}{q-1}\right)^{-1/3}$. 
This bound is polynomial in the size of the group when $n$ is a constant, but the dependency on the density deteriorates as the rank $n$ increases. 
Gowers' result has found various applications in theoretical computer science, e.g. to communication complexity \cite{gowers2015communication} and to questions related to matrix multiplication \cite{blasiak2023matrix}. 

As a further corollary of Theorem~\ref{thm:SLn-product-mixing}, we obtain a structural/stability version for Gowers' problem. 
%Specifically, if $A \subset \SLnq$ is a product free set of density $\ge q^{cn^2}$, then $A$ cannot be global, and therefore must be highly correlated with an umvirate.

\begin{cor}\label{cor:SLn-stability-Gowers}
There exists $c>0$, such that for any $n\in \N$ and any prime power $q$, the following holds. 
If $A\subseteq \SLnq$ is a product free set of density $\mu(A) \ge q^{-cn^2}$, then there exists a $t$-umvirate $U$, such that  $\frac{|A\cap U|}{|U|}\ge q^{ctn}\mu(A)$.
\end{cor}

\subsection{Approximate subgroups}

Let $K \in \N$. 
A set $A\subseteq \SLnq$ is said to be a $K$-\emph{approximate subgroup} if $A=A^{-1}$ and $|A^2| \le K |A|$.
The structure of approximate subgroups is well understood in the bounded rank regime (see e.g. \cite{breuillard2011approximate,breuillard2012structure,pyber2016growth,eberhard2021growth}), however the case where the rank $n$ is allowed to grow to infinity is completely open. 
We make the following step towards understanding the high rank regime. 

\begin{thm}\label{thm:approximate}
There exists $C>0$, such that for every $n\in \N$ and every prime power $q$, the following holds. 
If $A \subseteq \SLnq$ is a $K$-approximate subgroup, then there exists a good groumvirate $H$ of density $\mu(H) \ge \mu(A)^C$, such that $A$ is contained in the union of $\frac{K^5|A|}{|H|}$-cosets of $H$.
\end{thm}

We note that an analogue result for $A_n$ appears in Keevash and Lifshitz~\cite[Thm.~1.2]{keevash2023global}. 
Additionaly, note that when the factor $K$ in Theorem~\ref{thm:approximate} is not too large, then the union of cosets that is claimed to contain $A$ is not much larger than  $A$. 
Intuitively, this means that  whenever a large set $A$ is an approximate group there must be  an underlying groumvirate that explains this.

\subsection{Methods}

Our work relies on ideas of Sarnak and Xue~\cite{sarnak1991bounds}, which were later also used by Gowers~\cite{gowers} in study of product free sets, and on some refinments by Keevash, Lifshitz, and Minzer~\cite{keevash2022largest}. 

Write $L_0^2(G)$ for the space of functions on $G$ with $\bE[f] = 0$.
A key idea in \cite{sarnak1991bounds} is that if $G$ is a group and $T$ is a $G$-morphism on $L_0^2(G)$, then one can  upper bound the eigenvalues of $T$ by combining an upper bound on the trace of $T^*T$ with a lower bound on the minimal dimension of an eigenspace. 
The latter is always lower bounded by the minimal dimension of a nontrivial representations of $G$. 
Gowers called this minimal dimension, which we denote by $D(G)$, the \emph{Quasirandomness} of $G$, and proved that product free sets have density $\le D(G)^{-1/3}$. 
By \cite{landazuri1974minimal} for finite simple groups of Lie type of bounded rank, $D(G)$ is polynomial in the size of $|G|$. 
However, in the unbounded rank case, or for $G=A_n$, it is significantly weaker. 

In $A_n$, the above approach was refined by Eberhard~\cite{eberhard} and Keevash--Lifshitz~\cite{keevash2023global}. 
The latter paper obtained improved bounds by showing that for indicators of global sets, almost all of the Fourier mass is concentrated on the high dimensional representations. 
This was used to substantially improve Gowers' bound for global product free sets by Keevash and Lifshitz \cite{keevash2023global} to $e^{-O(D(A_n)^{1/3})}$. 

In order to show that global product free sets have their mass concentrated on the high dimensional representations, \cite{keevash2023global} followed Ellis, Friedgut, and Pilpel~\cite{ellis2011intersecting} and decomposed the space $L^2(A_n)$ into levels\footnote{\cite{ellis2011intersecting} actually worked with $S_n$, but the difference is insignificant.}, $L^2(A_n) = \bigoplus_{d=0}^n V_{=d}$, where each space $V_{=d}$ is an $A_n$-bimodule of $L^2(A_n)$ and the minimal dimensions of subrepresentation of $V_{=d}$ increases rapidly with $d$. 
They then used ideas from the theory of Boolean functions to show that when $d$ is small, the projection of indicators of global set onto $V_{=d}$ have negligible $L^2$-norm. 

In this paper we develop an analogue theory for $\SLnq$, namely we define a decomposition of $L^2(\SLnq)$ into a direct sum of $\SLnq$-bimodules $L^2(\SLnq)_{=d}$ and show that the dimension of the irreducible subrepresentations inside $L^2(\SLnq)_{=d}$ increase rapidly with $d$. 
Additionally, we show a `level $d$ inequality' which implies that indicators of global sets have small projections on spaces $L^2(\SLnq)_{=d}$ where $d$ is small.

\subsection{Levels and tensor rank on $L^2(\SLnq)$}

Our partition of $L^2(\SLnq)$ uses the idea of \emph{tensor rank} of representations, first defined by Gurevich and Howe~\cite{gurevich2021harmonic}. 
Their approach was a departure from the Harish-Chandra philosophy of cusp forms which, roughly speaking, classifies the set of irreducible representations in terms of the cuspidal representations, that are in a sense the largest ones.

Consider the permutation representation $\omega$ of $\SLnq$ on $L^2(\F{q}^n)$, given by $\omega(g)f(x) = f(g^{-1}x)$ for any $f\in L^2(\F{q}^n)$ and $g,x\in \SLnq$, which decomposes as $\omega = \mathbf{1} \oplus \omega_0$, where $\mathbf{1}$ and $ \omega_0$ are the trivial and smallest dimensional non-trivial irreducible representations of $\SLnq$. 
By \cite[Def.~3.1.1]{gurevich2021harmonic} an irreducible representation of $\SLnq$ is said to be of tensor rank $k$ if it  appears in $\omega^{\otimes k}$, the $k$-fold tensor product of $\omega$, but not in $\omega^{\otimes (k-1)}$. 
Denote by $(\reallywidehat{\SLnq})_{\otimes,k}$ the set of irreducible representation of $\SLnq$ of tensor rank $k$.
By \cite[Prop.~1.2.1]{gurevich2021harmonic}, the $(\reallywidehat{\SLnq})_{\otimes,k}$ for $k=0,\ldots,n$, form a partition of the irreducible representation of $\SLnq$.

Recall that an irreducible representation $\rho$ of a finite group $G$, is finite dimensional and unitary, and its matrix coefficients are functions of the form $\rho_{v,u} \in L^2(G)$, $\rho_{v,u}(g) = \langle \rho(g)v,u \rangle$, where $v,u\in V_\rho$. 
Denote by $L^2(G)_\rho$ the subspace spanned by the matrix coefficinets of $\rho$ in $L^2(G)$.
By the Peter-Weyl Theorem, 
\[
L^2(G)_\rho \cong \rho \otimes \rho^*, \qquad \forall \rho \in \widehat{G},
\]
as $G$-bimodules, and
\[
L^2(G) = \bigoplus_\rho L^2(G)_\rho, \qquad \rho \in \widehat{G}.
\]
a decomposition into irreducible $G$-bimodules.

\begin{defn} \label{def:level-L^2(SL(V))}
For $0\le d \le n$, denote by
\[
L^2(\SLnq)_{=d} = \bigoplus_\rho L^2(\SLnq)_\rho, \qquad  \rho \in (\reallywidehat{\SLnq})_{\otimes,d}
\]
the space spanned by matrix coefficients of irreducible representations of tensor rank $d$.
The partition of $L^2(\SLnq)$ into bimodules according to the levels/tensor rank is then $L^2(\SLnq) = \bigoplus_{d=0}^n L^2(\SLnq)_{=d}$.
For any $f\in L^2(\SLnq)$, let $f_{=d}$ be the projection $f$ onto the subspace $L^2(\SLnq)_{=d}$. 
If $f = f_{=d}$, then we say that $f$ is a \emph{level $d$ function}.
\end{defn}

Observe that the matrix coefficients $\omega_{v,u}$, of the permutation representation $\omega$ of $\SLnq$ on $L^2(\F{q}^n)$, are simply the indicators of the dictators $\d_{v,u}$. 
Hence $L^2(\SLnq)_{=1}$, the space of matrix coefficients of $\omega$, is therefore the space of linear combinations of dictators. 
Similarly, the matrix coefficients of the representation $\omega^{\otimes d}$, are degree $d$-monomials in the dictators $\mathcal{D}_{v,u}$, i.e. indicators of $d$-umvirates. 
This shows that the spaces of matrix coefficients of $\omega^{\otimes d}$ are exactly the polynomials of degree $d$ in the dictators. 
This yields an analytic interpretation of $L^2(\SLnq)_{=d}$, namely it is the space of polynomials of degree $d$ in the dicatators which are orthogonal to all degree $\le d-1$ polynomials in the dictators.

For the lower bound on the dimensions of irreducible representations inside $L^2(\SLnq)_{=d}$, we rely on the breakthrough work of Guralnick, Larsen, and Tiep~\cite[Thm.~1.3]{guralnick2020character}, who identified the representations of tensor rank $d$ and showed that the dimension of such a representation increases rapidly with $d$.

\subsection{Levels and degrees on $L^2(\l(V,V))$}

In order to obtain the level $d$ inequality we first prove a corresponding Theorem~in the bilinear scheme $\l(V,V)$. 
That space is also equipped with a natural decomposition into levels but it does not seem, at first look, to be related to the spaces $L^2(\SLV)_{=d}$. 
However we manage to bridge the two notions by using another deep Theorem~of Gurevich and Howe that turns out to releate the representation theoretic notion of tensor rank with the Boolean theoretic notion of a \emph{junta}. 

On the Boolean cube, a function $f \colon \{0,1\}^n \to \{0,1\}$ is said to be a $d$-\emph{junta} if it depends only on $d$ variables. 
Similarly, we say that a function $f$ on $\SLV \cong \SLnq$ is a $d$-junta if there exists a subspace $U \le V$ of dimensions $d$, such that $f(g)$ depends only on the restriction of $g$ to $U$. 
More generally, if $M$ is a $\SLV$-module, then we say that $f\in M$ is a $d$-\emph{junta} if there exists a $d$-dimensional subspace $U \le V$, such that $f$ is invariant under the action of the subgroup of $\SLV$ whose elements point-wise stabilize $U$ (note that the dictator functions $\d_{v,u}$ defined above are $1$-juntas). 
Gurevich and Howe showed that the tensor rank of an irreducible representation $M$ of $\SLV$ is the minimal $d$ for which $M$ contains a nonzero $d$-junta. 
We thus connect between the notions of level in the nonabelian $\SLV$ and the abelian $\l(V,V)$, as both spaces of level $d$ functions are spanned by $d$-juntas.   

As mentioned above, we obtain a level $d$ inequality over $\SLV$ using the result from \cite{ellis2022analogue} for the abelian group $\l(V,W)$. 
To state the connection between the two let us first briefly recall the abelian Fourier analysis on $\l(V,W)$, and using it we introduce the abelian notion of level/degree on $\l(V,W)$.

For a prime power $q = p^m$, define the homomorphism $\varphi \colon \F{q} \to \C^\times$, by $\varphi(x) = e^{\frac{2\pi i}{p} \sum_{i=0}^{m-1} x^{p^i}}$.
%Let $p$ be a prime, and let $\omega_p = e^{\frac{2\pi i}{p}}$ be a $p$-th root of unity. We now define a homomorphism $\varphi$ from $\F{p}$ to the unit circle in $\C$ by setting $\varphi(x) = \omega_p^x$. Let $q=p^i$ be a power of $p$.  Abusing notations, we extend $\varphi$ to $\F{q}$ by setting $\varphi(x) = \varphi(\mathrm{tr}(x))$, where we recall that the trace of $x \in \F{q}$ is the sum $\sum_{j=1}^ix^{p^j}$. 
The bilinear scheme $\l(V,W)$ is equipped with the characters $\{u_X\}_{X\in \l(W,V)}$ given by 
\[
u_X(A) = \varphi(\mathrm{tr}(XA)).
\] 
The characters $u_X$ are an orthonormal basis for $L^2(\l(V,W))$ and we write $\hat{f}(X) = \langle f,u_X\rangle$ for the \emph{Fourier coefficients} of $f$. 
The \emph{Fourier expansion} of $f$ is given by 
\[
f=\sum_{X\in \l(W,V)}\hat{f}(X)u_X.
\] 
%The \emph{level/degree} of a character $u_X$ is the rank of $X$. 

We now define the notion of Abelian level of functions on $\l(V,V)$ and on $\SLV$.

\begin{defn} \label{def:level-L^2(l(V,V))}
For any $f \in L^2(\l(V,V))$ and any $0\le d\le n = \dim V$, write 
\[
f^{=d} = \sum_{\mathrm{rank}(X) = d} \hat{f}(X) u_X, \qquad 
f^{\le d} = \sum_{\mathrm{rank}(X) \le d} \hat{f}(X) u_X.
\]
If $f = f^{=d}$ (resp. $f= f^{\le d}$), then we say that $f$ is of \emph{pure degree} $d$ (resp. of \emph{degree} $d$). 
Denote
\[
j \colon L^2(\SLV) \to L^2(\l(V,V)),\qquad j(f)(x) = \left\lbrace \begin{array}{cc} f(x) & x\in \SLV  \\ 0 & x\not\in \SLV \end{array} \right. .
\]
\end{defn}

We find the following remarkable connection between the nonabelian notion of tensor rank/level in $\SLV$ (Definition~\ref{def:level-L^2(SL(V))}) and the abelian notion of level/degree in $\l(V,V)$ (Definition~\ref{def:level-L^2(l(V,V))}).

\begin{lem} \label{lem:level-level}
Let $f \in L^2(\SLV)_{=d}$, i.e. $f$ is of (non-Abelian) level $d$.
Then 
\[ 
\| j(f)^{\le d} \|_2 \ge \frac{1}{4q} \|f\|_2. 
\]\gnote{I made a change here}
\end{lem}

\subsection{Bonami type and level inequalities}

We make use of the following Bonami type inequality for functions on $\l(V,W)$, which generalizes the Bonami type result of Ellis, Kindler, and Lifshitz~\cite{ellis2022analogue} from $4$-norms to $\ell$-norms, where $\ell$ is any power of $2$.

\begin{thm} \label{thm:hypercontractivity}
Let $f \in L^2(\l(V,W))$ be $(d,\epsilon)$-global of degree $d$, and let $\ell$ be a power of $2$. 
Then
\[
 \|f\|_\ell^\ell \le q^{200 d^2 \ell^2} \|f\|_2^2\epsilon^{\ell/2 - 1}.
\]
\end{thm}

From Theorem~\ref{thm:hypercontractivity} we obtain the following level $d$ inequality. 

\begin{thm} \label{thm:level-inequality}
Let $f\colon \l(V,W)\to \{0,1\}$ be $(d,\epsilon)$-global, and let $\ell$ be a power of $2$. 
Then
\[
\|f^{=d}\|_2^2 \le q^{460 d^2\ell}\bE[f]\epsilon^{1-2/\ell}.
\] 
\end{thm}

Theorem~\ref{thm:level-inequality} gives a level inequality on $\l(V,W)$ w.r.t. the Abelian level notion (Definition~\ref{def:level-L^2(l(V,V))}). 
We also prove in Theorem~\ref{thm:flexible-level-GLn} below a level inequality on $\SLV$ w.r.t. the non-Abelian level notion (Definition~\ref{def:level-L^2(SL(V))}).
More precisely, we give a bound for the weight that functions over $\SLV$ have on spaces of low non-Abelian level (i.e. over representations of low tensor rank). 
Theorem~\ref{thm:flexible-level-GLn} is obtained from Theorem~\ref{thm:level-inequality}, combined with the relation expressed in Lemma~\ref{lem:level-level} between the Abelian and non-Abelian notions of level.

\subsection{From level inequalities to growth} 
As mentioned above, the convolution estimate of Theorem~\ref{thm:SLn-mixing-global} is the main component in the proof of Theorem~\ref{thm:Bogolyubov}. 
Let us explain how it is obtained from our level $d$ inequality (Theorem~\ref{thm:level-inequality}). 
Let $A,B\subseteq\SLnq$ be two global subsets, let $f=\frac{1_A}{\mu(A)}, g=\frac{1_B}{\mu(B)} \in L^2(\SLnq)$ be their normaised indicators, and by abuse of notation, we identify them with $j(f), j(g) \in L^2(\l(\F{q}^n,\F{q}^n))$. 
Consider the decomposition of $g$ into its (non-Abelian) level componenets, as defined in Definition~\ref{def:level-L^2(SL(V))},
\[
g = \sum_{d=0}^n g_{=d} = 1 + \sum_{d=1}^n g_{=d}.
\] 
Let $T_f$ denote the operator on $L^2(\SLV)$ defined for any $h \in L^2(\SLV)$ by
\[
T_f h = f * h
\]
Expanding the convolution we thus get
\[
T_f g = 1 + \sum_{d=1}^n T_f g_{=d}, 
\]
and our goal reduces to showing that $\|T_f g_{=d}\|_2$ is small for each $d \ge 1$. 
This is acheived by both bounding the norm of $g_{=d}$ using the (non-Abelian) level $d$ inequality (Theorem~\ref{thm:flexible-level-GLn}) and bounding the norm of the operator $T_f$ when restricted to functions of (non-Abelian) level $d$. 
The latter bound is stated by the following theorem. 

\begin{thm}\label{thm:convolution-degree}
There exists $c>0$, such that for any $n \in \N$ and any prime power $q$, the following holds. 
Let $A\subseteq \SLnq$ be a global set with $\mu(A) > q^{-c^2n^2}$, $f=\frac{1_A}{\mu(A)}$ and $1 \le d \le n$. 
Then for every $h \in L^2(\SLnq)_{=d}$,
\[
\|T_f h\|_2\le q^{-cdn}\|h\|_2.
\]
\end{thm}

A key observation for the proof of Theorem~\ref{thm:convolution-degree} is when restricted to functions of level $d$, the operator $T_f$ is equal to the opertor $T_{f^{=d}}$ which applies convolution with the $d$-level part of $f$. 
The bound on the norm of this operator is then obtained by applying two inequalities: 
The first is the level $d$ inequality, again, that bounds the norm of $f^{=d}$ for small values of $d$, and the other is a result of Guralnick, Larsen, and Tiep~\cite{guralnick2020character}, which gives a lower bound showing that every subrepresentation of $V_{=d} = L^2(\SLnq)_{=d}$ has dimension $q^{\Theta(nd)}$. 
The lower bound on the dimension is useful to bound the norm of $T_{f^{=d}}$ for larger values of $d$. 

Let us explain in more detail how Theorem~\ref{thm:convolution-degree} is obtained, following a similar framework as is used by  Keevash, Lifshitz, and Minzer~\cite{keevash2022largest}. 
In order to show that $f*g$ is close to $1$ in $L_2$ norm let us write 
\[
\|T_f \|_{V_{=d}} = \sup_{0 \ne h\in L^2(\SLnq)_{=d}} \frac{\|T_f h\|_2}{\|h\|_2}.
\] 
In other words $\|T_f\|_{V_{=d}}$ is the operator norm of the restriction of $T_f$ to $L^2(\SLnq)_{=d}$. 
As mentioned above, $\|T_f\|_{V_{=d}}=\|T_{f_{=d}}\|_{V_{=d}}$, and thus it is easy to see that $\|T_f\|_{V_{=d}}^2$ is equal to the maximal eigenvalue of the self adjoint operator $S := T_{f_{=d}}^* T_{f_{=d}}$ acting on $L^2(\SLnq)_{=d}$. 
On the other hand, it turns out that the trace of $S$ is equal to the $L^2$-norm of $f_{=d}$. 
Moreover\gnote{!}, since the operator $S$  commutes with the action of $\SLnq$ from the right its eigenspaces are subrepresentations of $L^2(\SLnq)_{=d}$ and therefore the maximal eigenvalue of $S$ can be upper bounded by the ratio between its trace, namely $\|f_{=d}\|_2^2$, and the minimal dimension of a irreducible representation of tensor rank $d$. 
Combining this with the level $d$ inequality for $f$ and the dimension lower bound for tensor rank $d$ irreducible representations yields an upper bound on  $\|T_f\|_{V_{=d}}$. 
We note that the idea of combining the trace method with representation theoretic data appeared first in the works of Sarnak and Xue~\cite{sarnak1991bounds}. 

\subsection{Relations with previous works} 
This paper can be considered as a continuation of a recent line of works that extend results about Boolean valued functions over the Boolean cube to non abelian settings. 
The first such result of this kind, as far as is  known to the authors, is that of Ellis, Filmus and Friedgut \cite{ellis2015stability}, who studied stability versions of bounds by Ellis, Friedgut and Pilpel \cite{ellis2011intersecting} on intersecting families of permutations. 

This line of study received further motivation from computer science, specifically from the study of the so-called $2$ to $2$ conjecture~\cite{dinur2018non,khot2018pseudorandom}. 
These works focused on function over $k$-dimensional subspaces of $\F{2}^n$. 
They defined a notion of pseudornadomness, which is analogous to our notion of globalness, and showed that the 4-norm of global sets is small compared to their 2-norm. 
The original proof that appears \cite{khot2018pseudorandom} is a breakthrough, but it is complicated and quantitatively far from optimal. 

Keevash, Long, Lifshitz, and Minzer~\cite{keevash2021global} then showed how to deduce level $d$ inequalities for global functions from a Theorem~called `hypercontractivity for global functions' in the setting of the $p$-biased cube. One of their ideas is to use iterated Laplacians and derivatives to measure the globalness of a functions in an analytic way. 
Ellis, Kindler, and Lifshitz~\cite{ellis2022analogue} then imported this idea and applied it to the bilinear scheme. 
They defined analogue notions of Laplacians and derivatives, and used these to give a much simpler proof of the level $d$-inequality for global sets of Khot, Minzer and Safra~\cite{khot2018pseudorandom}. 
Moreover, their proof  has two advatages that are crucial for the applications of this paper: 
Their result is quantitatively sharp, and their notions of Laplacians and derivatives can be used to obtain a Bonami type inequality for global sets from $L_2$ to $L_{2^i}$ for any $i$, unlike the earlier more direct approach that seems to only work when $i=2$. 

The ideas of reducing Bonami type inequalities in the non-Abelian setting from the Abelian setting is due to Filmus, Kindler, Lifshitz, and Minzer~\cite{filmus2020hypercontractivity}. 
In Ellis, Kindler, Lifshitz, and Minzer~\cite{Ellis2023product} this idea was used for proving a hypercontractive estimate for all compact Lie group of sufficiently high rank.

\subsection{Future work}
We hope that our results find future applications. 
Indeed, the preprint of our work was already found useful for applications in extremal combinatorics. 
Kelman, Lindzey, and Sheinfeld \cite{Kelman2023tint} applied our Bonami type Theorem~\ref{thm:hypercontractivity} to obtain new bound for Erd\H{o}s--Ko--Rado type theorems for matrices. 
It is worth mentioning that a weaker variant of the Bogolyubov lemma appeared in the Helfgott--Seress~\cite{helfgott2014diameter} and perhaps Theorem~\ref{thm:Bogolyubov} will play a similar role in the future for the analogue problem for $\SLnq$.

\subsection{Structure of the paper}
In Section~\ref{sec:preliminaries} we recall results from \cite{ellis2022analogue} and we also set two notions of globalness, namely, globalness and small generalized influences. 
In Section~\ref{sec:influence} we show that these are essentially equivalent. 
In Section~\ref{sec:proof-1.13} we prove Theorem~\ref{thm:hypercontractivity}. 
In Section~\ref{sec:proof-1.14} we prove Theorem~\ref{thm:level-inequality}. 
In Section~\ref{sec:proof-1.12} we prove Lemma~\ref{lem:level-level}. 
In Section~\ref{sec:proof-1.15} we prove Theorem~\ref{thm:convolution-degree}. 
In Section~\ref{sec:proof-1.3,6,7,8} we prove Theorems~\ref{thm:SLn-mixing-global} and \ref{thm:SLn-product-mixing} as well as Corollaries~\ref{cor:nikolov-pyber} and \ref{cor:SLn-stability-Gowers}. 
Finally, in Section~\ref{sec:proof-1.2,9} we prove Theorems~\ref{thm:Bogolyubov} and \ref{thm:approximate}

\section{Preliminaries from \cite{ellis2022analogue}} \label{sec:preliminaries}

\subsection{Iterated Laplacians and derivatives}

The iterated Laplacians in $\l(V,W)$ (or simply Laplacians) were defined in the context of product spaces by Keevash, Long, Lifshitz, and Minzer~\cite{keevash2021forbidden}. 
They were then extended by \cite{ellis2022analogue} to $\l(V,W)$. 
The rough idea is as follows. 

The discrete derivatives of a function on the Boolean cube $f \colon \{-1,1\}^n \to \R$ are given by $D_i(f) = \frac{f_{i\to 1} - f_{i\to -1}}{2}$.
One can then form the iterated derivatives by setting $D_{\{i,j\}}(f) = D_iD_j(f)$ and the derivative $D_{S}(f)$ is then defined by repeatedly applying the operators $D_i$ over all $i\in S$.
This notion of a discrete derivative does not extend immediately even to other product spaces, such as $\F{p}^n$. 
The idea of Keevash, Lifshitz, Long, and Minzer~\cite{keevash2021forbidden} was to define the derivatives as the restrictions of the Laplacians. 
In the case of the Boolean cube a function $f\colon \{-1,1\}^n \to \R$ can be expanded in terms of its \emph{Fourier} expansion 
\[
f=\sum_{S\subseteq [n]} \hat f(S)\chi_S,
\] 
where $\chi_S(x) = \prod_{i\in S}x_i$. 
The \emph{iterated} Laplacians of $f$ are then given by 
\[
L_S(f) = \sum_{T\supseteq S}f^{=T}.
\] 
It was then observed in \cite{keevash2021forbidden} that the derivatives can be recovered from the Laplacians by plugging in an arbitrary $x\in\{-1,1\}^S$ in the $S$ coordinates and looking at the restricted function $L_S(f)_{S\to x}\colon \{-1,1\}^{S^c}\to \R$. 
The function $L_S(f)_{S\to x}$ is equal to $\chi_S(x)D_S(f)$. While the notion of discrete derivative does not extend well to other product spaces, the notion of Laplacian is much more flexible and can be defined in various settings. 
In \cite{keevash2021forbidden} Keevash, Lifshtz, Long, and Minzer managed to extend the Laplacians to arbitrary product spaces and defined the derivatives as their restrictions. 
In \cite{ellis2022analogue}, Ellis, Kindler and Lifshtz followed a similar route by giving the following definition for the Laplacians and then defining their derivatives as their restrictions. 

\begin{defn}
Let $V_1 \le V$, $W_1 \le W$ and $T \in \l(V,W)$.
The \emph{Laplacian} $L_{V_1,W_1}$ and the {\em derivative} $D_{V_1,W_1,T}$ are the operators on $L^2(\l(V,W))$, defined for any $f \in L^2(\l(V,W))$ by 
\[
L_{V_1,W_1}(f) := \sum_{X\in\l(W,V):\, \im(X) \supseteq V_1,\, X^{-1}(V_1)\subseteq W_1}\hat{f}(X)u_X,
\]
\[
D_{V_1,W_1,T}(f) := (L_{V_1,W_1}(f))_{(V_1,W_1)\to T}.
\] 
Call $i=\dim(V_1)+\codim(W_1)$ the order of $L_{V_1,W_1}$ and $D_{V_1,W_1,T}$.
For brevity, we write $D_{V_1,W_1}: = D_{V_1,W_1,0}$. 
\end{defn}

In \cite{ellis2022analogue} the name `derivatives of order $i$' for the operators $D_{V_1,W_1,T}$ was justified. 
They showed that it sends functions of pure degree $d$ to functions of pure degree $d-i$. 

\begin{lem} \cite[Lem.~35]{ellis2022analogue} \label{lem:EKL-35} 
Let $V_1\le V$, $W_1\le W$, $i=\dim(V_1)+\mathrm{codim}(W_1)$, $d \in \N$ and $f \in L^2(\l(V,W))$.
Then
\[
D_{V_1,W_1,T}(f^{=d}) = \left( D_{V_1,W_1,T}(f) \right)^{=d-i}.
\]
\end{lem}

They also showed that derivatives behave well with respect to compositions and the composition of a derivative of order $i$ with a derivative of order $j$ is a derivative of order $i+j$. 

\begin{prop}\cite[Prop.~38]{ellis2022analogue} \label{prop:EKL-38} 
Let $V_2 \le V_1\le V$, $W_1\le W_2\le W$, $T \in \l(V,W)$ and $S\in \l(V/V_2,W_2)$. 
Then
\[
D_{V_1/V_2,W_1,S}\circ D_{V_2,W_2,T} = D_{V_1,W_1,T+S}.
\]
\end{prop}

\subsection{Influences}

Recall from Definition \ref{def:global} that we say that $f \in L^2(\l(V,W))$ is $(d,\epsilon)$-global if $\|f_{(V_1,W_1)\to T}\|_2^2\le\epsilon$ for each $V_1 \le V$ and $W_1 \le W$ with $\dim(V_1)+\text{codim}(W_1) = d$, and each $T \in \l(V,W)$.

\begin{defn} \label{def:influences} 
Let $V_1 \le V$, $W_1 \le W$ and $T \in \l(V,W)$.
The {\em influence of $(V_1,W_1)$ at $T$}, is the functional on $L^2(\l(V,W))$, defined for any $f \in L^2(\l(V,W))$ by 
\[
I_{V_1,W_1,T}(f) := \|\, D_{V_1,W_1,T}(f)\|_2^2.
\]
We say that $f \in L^2(\l(V,W))$ has {\em $(d,\epsilon)$-small generalized influences} if $I_{V_1,W_1,T}(f)\le\epsilon$ for each $V_1 \le V$ and $W_1 \le W$ with $\dim(V_1)+\codim(W_1)\le d$, and each $T\in\l(V,W)$.
\end{defn}

Ellis, Kindler, and Lifshitz~\cite{ellis2022analogue} showed that globalness implies that $f^{=d}$ has small
generalized influences.

\begin{prop} \cite[Prop.~63]{ellis2022analogue} \label{prop:EKL-63} 
Let $d\in \N$, $\epsilon>0$ and $f\in L^2(\l(V,W))$. 
If $f$ is $(d,\epsilon)$-global, then $f^{=d}$ has $(d,q^{10d^2}\epsilon)$-small generalized influences.
\end{prop}

They then proved the following hypercontractive inequality.

\begin{thm} \cite[Cor.~65]{ellis2022analogue} \label{thm:EKL-65} 
Suppose that $f \in L^2(\l(V,W))$ is a function of degree at most $d$ that has $(d,\epsilon)$-small generalized influences. 
Then 
\[
\|f\|_{4}^{4}\le q^{103d^2}\epsilon\|f\|_2^2.
\]
\end{thm}

\subsection{The averaging operator $\mE_v$}

In the Boolean cube, the Laplacian $L_i(f)$ has a combinatorial interpretation. 
Let $E_i(f)(x)$ be the expectation of $f(x')$, where $x'$ is obtained from $x$ by resampling its $i$th coordinate from $\{-1, 1\}$ uniformly at random. 
Then $L_i(f) = f - E_i(f)$. 
There is no completely straightforward way to generalize the averaging operator $E_i(f)$. 
Indeed, given a linear map $A$, one cannot simply change its value on a vector $v$ without affecting its values on other vectors. 
A possible attempt to generalize the Laplacian is to complete $v$ to a basis $v=v_1,v_2,\ldots,v_n$ of $V$, leaving the value of $v_i$ as it is for all $i\ge2$, while resampling the value of $v$. 
The problem with this approach is that different choices of the vectors $v_2,\ldots,v_n$ yield different operators.
\cite{ellis2022analogue} gave the following combinatorial version of the Laplacian by setting it to be the average of all such operators.

\begin{defn}\label{def:B}
Given a subspace $V'\le V$, we define the linear operator $\mathfrak{e}_{V/V'} \colon L^2(\l(V,W)) \to L^2(\l(V,W))$ by
\[
(\mathfrak{e}_{V/V'}(f))(A) := \underset{B \in \l(V/V',W)} \bE f(A+B) \qquad \forall A \in \l(V,W),
\]
where the expectation is (as the notation suggests) over a uniform random element of $\l(V/V',W)$. 

Given $0 \ne v \in V$, we define the linear operator $\mE_v \colon L^2(\l(V,W)) \to L^2(\l(V,W))$ by
\[
\mE_v(f):=\underset{V' \notni v}{\bE}[\mathfrak{e}_{V/V'}(f)],
\]
where the expectation is over a uniformly random subspace $ v \notin V' \subseteq V$ of codimension one.

Given $0 \ne v \in V$, we define the {\em combinatorial Laplacian} $\mathfrak{L}_v \colon L^2(\l(V,W)) \to L^2(\l(V,W))$ by
\[
\mathfrak{L}_v(f) := f-\mE_v(f) \quad \forall f \in L^2(\l(V,W)).
\]
If $U$ is the one-dimensional subspace spanned by $v$, then we may write $\mE_U$ and $\mathfrak{L}_U$ instead of $\mE_v$ and $\mathfrak{L}_v$, respectively. 
(The operator $\mE_U$ is easily seen to be independent of the choice of the generator $v$.) 
\end{defn}

We note that the combinatorial Laplacian $\mathfrak{L}_v$ is the Laplacian of the Markov chain on $\l(V,W)$ where at each step, we replace a matrix $A$ with $A+B$, where $B$ is a uniform random element of $\l(V/V',W)$ and $V'$ is a uniform random codimension-one subspace of $V$ that does not contain $v$ (the random choices being independent of all previous steps). 
The following formulas for the Fourier expansion were obtained by \cite{ellis2022analogue}.

\begin{lem} \cite[Lem.~42]{ellis2022analogue} \label{lem:EKL-42}
For any $X \in \l(W,V)$, we have
\[
\mathfrak{e}_{V/V'}(f) = \sum_{X:\,\Image(X)\subseteq V'}\hat{f}(X)u_{X}.
\]
\end{lem}

By averaging the above, they obtained the following result.

\begin{lem} \cite[Lem.~43]{ellis2022analogue} \label{lem:EKL-43}
For any $f \in L^2(\l(W,V))$ and $v \in V \setminus \{0\}$, we have
\[
\mE_v(f) = \sum_{X \in \l(W,V):\ v \notin \Image(X)} q^{-\rank(X)} \hat{f}(X)u_X.
\]
\end{lem}

\subsection{The dual operators $\mE_{W'}$}

For $f \in L^2(\l(V,W))$, define $f^{*}\in L^2(\l(V^*,W^*))$ by $f^*(A)=f(A^{*})$ for each $A \in \l(W^*,V^*)$.
All of the above notions for $f^*$ correspond to dual notions for the function $f$.

Given a subspace $W'\le W$ of codimension $1$, we define the linear operator 
\[
\mE_{W'}:L^2(\l(V,W)) \to L^2(\l(V,W))
\]
as follows. 
We let $\varphi \in W^*$ with $\varphi\ne0$ and $\varphi\left(W'\right)=0$, and set 
\[
\mE_{W'}(f) = (\mE_{\varphi}[f^*])^* \qquad \forall f \in L^2(\l(V,W)).
\]
Dually to Lemma \ref{lem:EKL-43}, we obtain

\begin{lem} \label{lem:dual-EKL-43}
For any $f \in L^2(\l(V,W))$ and any codimension-one subspace $W'$ of $W$, we have
\[
\mE_{W'}(f) = \sum_{X\in \l(W,V):\,\mathrm{Ker}(X)+W'=W}q^{-\mathrm{rank}(X)}\hat{f}(X)u_{X}.
\]
\end{lem}

\subsection{Combinatorial interpretation for the Laplacian of functions of pure degree $i$}

While the Laplacian does not have a nice combinatorial interpretation in terms of averaging operators for general functions, it does have one when $f$ is of pure degree $i$.  

\begin{lem} \cite[Lem.~59]{ellis2022analogue} \label{lem:EKL-59} 
Let $U$ be either a 1-dimensional subspace of $V$ or a subspace of $W$ of codimension 1, and let $i \in \N \cup \{0\}$. 
Then we have 
\[
\mathfrak{L}_U[f^{=i}]=f^{=i}-q^i\mE_U[f^{=i}].
\]
\end{lem}

A slightly messier combinatorial interpretation of the Laplacian was given in \cite{ellis2022analogue}, which works when $f=f^{=i}+f^{=i-1}$, namely when it is `almost pure degree'.

\begin{lem} \cite[Lem.~60]{ellis2022analogue} \label{lem:EKL-60}
Let $U$ be either a 1-dimensional subspace of $V$ or a subspace of $W$ of codimension 1, and let $i \in \N$. 
Write $\mT = \mT_{i,U}:L^2(\l(V,W)) \to L^2(\l(V,W))$ for the operator defined by 
\[
\mT f := f-(q^i+q^{i-1})\mE_{U}(f)+q^{2i-1}\mE_U^2(f)\quad \forall f \in L^2(\l(V,W)).
\]
Then for all $f \in L^2(\l(V,W))$ we have
\[
\mathfrak{L}_U[f^{=i}]=(\mT(f))^{=i}
\]
and
\[
\mathfrak{L}_U[f^{=i-1}]=(\mT(f))^{=i-1}.
\]
\end{lem}

We have the following lemma from \cite{ellis2022analogue} that describes the behavior of the restrictions of the characters. 
It can be used to compute the Fourier expansion of the derivatives of a function with a given Fourier expansion. 

\begin{lem} \cite[Lem.~25]{ellis2022analogue} \label{lem:EKL-25} 
Let $V_1\le V$, let $W_1\le W$, let $X\in\l(W,V)$, and let $Y=X\left(W_1,V/V_1\right)$, i.e.\ $Y$ is the linear map obtained by restricting the domain of $X$ to $W_1$, and then composing on the right with the quotient map $V \to V/V_1$. 
Then 
\[
(u_X)_{(V_1,W_1)\to T}=u_X(T)u_Y.
\]
\end{lem}

\section{Small generalized influences imply globalness} \label{sec:influence}

Proposition~\ref{prop:EKL-63} above, which was proved in \cite{ellis2022analogue}, shows that globalness implies small generalized influences. 
In this section we show that the converse also holds, namely that if a function $f$ of degree $d$ has $(d,\epsilon)$-small generalized influences, it must also be $(r,\epsilon')$-global for some $\epsilon'=\epsilon'(r,d,\epsilon)$ (see Proposition~\ref{prop:influence-global}).

Our idea is to argue inductively that each derivate of $f$ is global, and then to apply Lemma \ref{lem:EKL-43} to express the restriction of $f$ as a linear combination of a restriction of a derivative of $f$ and a restriction of $\mE_U(f)$. 
We then argue via Jensen's inequality that the corresponding restriction of  $\mE_U(f)$ is small by induction on $r$. For that purpose we need some observations about the averaging operator, which we make below.

\begin{defn}
For $v\in V$, we write $\mB_v\in \l(V,W)$ for the uniform distribution over $w\otimes \varphi$, where $w,\varphi$ are chosen independently and $w$ is uniformly random in $W$ and $\varphi$ is uniformly random among the functionals in $V^*$ that send $v$ to 1. 
Here we use the identification between $W\otimes V^*$ and $\l(W,V)$. 
\end{defn}

\begin{lem}
Let $f\in L^2(\l(V,W))$ and $A\in \l(V,W)$. 
Then
\[
\mE_v(f)(A) = \bE_{B\sim \mB_v}[f(A+B)].
\]
\end{lem}

\begin{proof}
We wish to show that the following two distributions are the same. 
One is $\mB_v$ and the other distribution is obtained by choosing a random $V'$ with $V' + \mathrm{Span}(v) = V$, and then setting $B:V\to W$ by letting $B(V')=0$, and defining $Bv$ to be a uniformly random vector $w\in W$. 
Indeed, if in the process for choosing $B$ we let $\varphi$ be the functional sending $V'$ to 0 and $v$ to $1$ and use the same $w$, then $B = \varphi \otimes w$. 
This completes the proof as $\varphi$ is in bijection with its kernel $V'$.
\end{proof}

We will need to understand the distribution of $A+B,$ where $A\sim \l(V/V', W')$ and $B\sim \mB_v$ for $v\in V'$.

\begin{lem}\label{lem:distribution}
Let $v\in V'$, $A\sim \l(V/V',W')$ and $w \otimes \varphi \sim \mB_v.$ Condition on the kernel $V''$ of $\varphi|_{V'}$ and on  $W''= W' + \mathrm{Span}\{w\}$.  
Then under the conditioning on $V'',W''$ the matrix $A + \varphi\otimes w $ is uniformly distributed in $\l(V/V'',W'')$ in the case where $W''=W'$, and the conditioning where $W''\ne W'$ it is uniformly distributed inside the set of all maps in $\l(V/V'', W'')$ that send $v$ outside of $W'$.   
\end{lem}

\begin{proof}
Condition on $V'',W''$. 
We consider first the case where $W''\ne W'$. 
Let $\mB_1$ be an ordered basis for $V/V''$ containing $v$ as its first vector. 
Let $\mB_2$ be an order basis of $W''$ containing a basis of $W'$ as its last vectors. 
When writing $A$ as a matrix with respect to the bases $\mB_1,\mB_2$ we get a matrix of the form 
$\left(\begin{array}{cc}
0 & 0\\
0 & \tilde{A}
\end{array}\right)$
whose first row and column is zero and $\tilde{A}$ is a uniformly random matrix. 
Now $B = w\otimes \varphi$ and the conditioning implies that $w$ is uniformly random on $W''\setminus W'$ and $\varphi$ is a uniformly random functional that sends $v$ to $1$. 
With respect to our bases we obtain that $w$ is a random vector under the conditioning $w_1\ne 0$ and $v$ is a random vector under the conditioning $v_1 = 1$. 
This easily implies that the first row and column of their tensor product are uniformly random under the conditioning that $(A+B)_{11} = v_1w_1\ne 0$. 
This completes the proof of this case as the condition $a_{11}\ne 0$ is equivalent to $(A+B)v\notin W'$. 
In the case where $W''=W'$ we can define a basis $\mB_1$ similarly and we obtain that $A$ is random on $\mB_1\setminus\{v\}$ and sends $v$ to 0, while $B$, which is independent of $A$, sends $v$ to a uniformly random vector in $W'$. 
This completes the proof.
\end{proof}

\begin{lem}\label{lem:E_U-global}
Let $d\in \N$ and let $T\in \l(V,W)$. 
Let $U$ be either a 1-dimensional subspace of $V$ or a subspace of $W$ of codimension 1, and let $f$ be of degree $d$. 
Suppose that $f$ is $(r,\epsilon)$-global, then $\mE_U(f)_{U\to T}$ is $(r,2\epsilon)$-global.
\end{lem}

\begin{proof}
Without loss of generality, we may assume that $U\le V$ (Otherwise if $U\le W$ we can view $f$ as a function on $\l(W^*, V^*)$). 
Let $V'\le  V, W'\le W$ be with $V'\ge U$ and such that $r = \dim(V'/U) + \codim(W'$  $(V'/U,W')$. %for $\mE_{U}(f)_{U\to T}$. 
Using translation, we may assume that $T=0$, and thus upper bound the $2$-norm of $\mE_{U}(f)_{(V',W')\to S}$ only when $S=0$. 
We now apply Cauchy--Schwarz to have
\[    
\bE_{A\in \l(V/V',W')} \bE_{B \sim \mB}^2[f(A+B)]\le \bE_{A,B}[f(A+B)^2].
\]
Let $V'',W''$ be as in Lemma \ref{lem:distribution}. 
Then when conditioning on $V'',W''$ we obtain that either $A+B$ is uniformly distributed in $\l(V/V'',W'')$ or $A+B$ is uniformly distributed in a subset of density $1-\frac{1}{q}$ of all elements sending $v$ to $W''\setminus W'$.  
In the former case we have 
\[
\bE[f(A+B)^2 | V'',W''] =  \| f_{(V'',W'')\to 0}\|_2^2 \le \epsilon.
\]
In the latter case we have 
\[
\bE[f(A+B)^2 | V'',W''] = \frac{q}{q-1}\| 1_{Av\notin W'} f_{(V'',W'')\to 0}\|_2^2 \le \frac{q}{q-1}\epsilon.
\]
The Theorem~follows by averaging over $V'',W''$.
\end{proof}

\begin{lem}\label{lem:derivative-global}
Let $f$ be a function of pure degree $d$. 
Suppose that $D_{U,T}(f)$ is $(r-1,\epsilon_1)$-global for all $U\subseteq V$ of dimension 1 and for all $U\subseteq W$ of codimension 1. 
Suppose additionally that $f$ is $(r-1,\epsilon_2)$-global. 
Then $f$ is $(r, 2 \epsilon_1 + 4 \cdot q^{2d} \epsilon_2)$-global.
\end{lem}

\begin{proof}  
Let $V' \le V$ and $W' \le W$ be such that $r = \dim(V') + \codim(W')$, and let $T\in \l(V,W)$. 
We show that $\|f_{(V',W')\to T}\|_2^2 \le \frac{1}{2}q^{10dr}\epsilon$. 
The case where $r=0$ follows from the fact that $f$ is the only $0$ derivative of $f$ and the only $0$-restriction of $f$. 
It therefore remains to consider the case where either $V'\ne 0$ or $W'\ne W$. 
We suppose without loss of generality that $V'\ne 0$. 
(Otherwise we can switch to the dual function on $\l(W^*, V^{*})$). 
Let $U\le V'$ be of dimension 1.
By Lemma \ref{lem:EKL-59}, we may write 
\[
f = L_U(f) + q^d \mE_U f.
\]
We now upper bound  
\[
\| f_{(V',W')\to T}\|_2  \le \|(L_U(f))_{(V',U') \to T}\|_2 + q^d\|\mE_U(f)_{(V',U')\to T}\|_2,
\]
which yields 
\begin{equation}\label{eq:lap-eu}
\| f_{(V',W')\to T}\|_2^2  \le 2\|(L_U(f))_{(V',U') \to T}\|_2^2 + 2\cdot q^{2d}\|\mE_U(f)_{(V',U')\to T}\|_2^2.
\end{equation}   
The first $r$-restriction above, $(L_U(f))_{(V',U')\to T}$, is an $(r-1)$-restriction of $D_{U,T}$. 
This implies that 
\[
\|(L_U(f))_{(V',U') \to T}\|_2^2 \le \epsilon_1.
\]
The second $r$-restriction $\mE_U(f)_{(V',U')\to T}$ is an $(r-1)$ restriction of $\mE_U(f)_{U\to T}$. 
By Lemma \ref{lem:E_U-global}, the function $\mE_{U}(f)_{U\to T}$ is $(r-1, 2\epsilon_2)$-global. 
This shows that 
\[
\|\mE_U(f)_{(V',U')\to T}\|_2^2\le 2\epsilon_2.
\]
Plugging our upper bounds in (\ref{eq:lap-eu}) completes the proof.
\end{proof}

\begin{prop}\label{prop:influence-global}
Suppose that $f$ is of degree $d$ and has $(d,\epsilon)$-small generalized influences, then it is $(r,q^{10dr}\epsilon)$-global for any $r\ge d$.
\end{prop}

\begin{proof}
Our proof is by nested induction. 
The primary assumption is on $d$, and simultaneously for all $r$. 
The  inner induction is on $r$, and is applied when $d$ is viewed as fixed. 
As the base of the induction, we note that the lemma is trivial when either $r$ or $d$ is $0$.

By Lemma \ref{lem:EKL-35} for each $i\le d$ the function $f^{=i}$ has $(d,\epsilon)$-small generalized influences, and therefore also $(i,\epsilon)$-small generalized influences. 
We then get by the outer inductive hypothesis that for all $i <  d$ the function $f^{=i}$ is $(r, q^{10ri}\epsilon)$-global.
Below we show that $f^{=d}$ is $(r,\epsilon_d)$-global for $\epsilon_d \eqdef \frac{1}{4}q^{10rd}\epsilon$. 
This will allow us to use the fact that each restriction of $f$ is the sum of the corresponding restrictions of the pure degree parts $f^{=i}$. 
This in turn will allow us to apply the triangle inequality to obtain that $f$ is $(r,\epsilon')$-global for 
\[
\epsilon' := \left(\sqrt{\epsilon} + \sqrt{q^{10r}\epsilon} + \ldots + \sqrt{q^{10r(d-1)}\epsilon} +\sqrt{\epsilon_d}\right)^2 \le 2 \epsilon_d + 4q^{10 r(d-1)}\epsilon \le 2\epsilon_d + \frac{1}{2}q^{10 rd}\epsilon.
\]  
Hence, once we show that the function  $f^{=d}$ is  $(r,\epsilon_d)$-global our proof will be completed. 
For simplicity of notation we now assume that $f$ is of pure degree $d$ namely $f=f^{=d}$) that has $(r,\epsilon)$-small generalized influences and show that it is $(r,\epsilon_d)$-global. 
    
By the inner induction hypothesis, the function $f$ is $(r-1, q^{10d(r-1)}\epsilon)$-global. 
Moreover the function $D_{U,T}(f)$ has $(d-1,\epsilon)$-small generalized influences as each derivative of $D_{U,T}(f)$ is also a derivative of $f$ by Proposition~\ref{prop:EKL-38}. 
This allows us to apply the outer induction hypothesis for $D_{U,T}(f)$ and obtain that $D_{U,T}(f)$ is $(r-1, q^{10(r-1)(d-1)}\epsilon)$-global.

We therefore obtain by Lemma \ref{lem:derivative-global} that $f$ is $(r,\epsilon'')$-global for 
\[
\epsilon'' = 2\cdot q^{10(r-1)(d-1)}\epsilon + 4 \cdot q^{2d}\cdot q^{10d(r-1)} \epsilon.
\] 
This completes the proof as 
\[
\epsilon''\le \frac{1}{4}q^{10dr}\epsilon=\epsilon_d.
\]
\end{proof}

\section{Bonami type inequalities} \label{sec:proof-1.13}

In this section our goal is to prove Theorem~\ref{thm:hypercontractivity}. 
We first use our Bonami-type lemma to show that if a function $f$ of degree $d$ is $(d,\epsilon)$-global, then its square $f^2$ is $(2d,\epsilon')$-global for an appropriate value of $\epsilon'$. 
This allows us to iteratively use the $4$ vs. $2$ Bonami type inequalities from Corollary \ref{thm:EKL-65} to upper bound the $\ell$-norm of a $d$-degree function $f$, by inductively upper bounding the $\ell/2$-norm of $f^2$. 
Equipped with this $\ell$-norm Bonami type inequality, we then obtain a level $d$ inequality that bounds the level $d$ weight of Boolean valued functions. 

\begin{lem}\label{lem:global-square}
If $f\in L^2(\l(V,W))$ is $(d,\epsilon)$-global of degree $d$, then $f^2$ is $(2d, q^{144d^2}\epsilon^2)$-global. 
\end{lem}

\begin{proof}
If $f$ is a $(d,\epsilon)$-global function, then by Proposition~\ref{prop:EKL-63} each $f^{=i}$ has $(d, q^{10i^2}\epsilon)$-small generalized influences. 
By Lemma \ref{lem:EKL-35} we obtain that $f$ has $(d, \sum_{i=0}^{d}q^{10 i ^2}\epsilon)$-small generalized influences. 
This shows that $f$ has $(d,q^{11 d^2}\epsilon)$-small generalized influences. 
By Proposition~\ref{prop:influence-global} the function $f$ is $(3d, q^{41d^2}\epsilon)$-global. 
This implies that for each $i\le 2d$ each $i$-restriction of $f$ is a $(d,q^{41 d^2}\epsilon)$-global function of degree $\le d$.
By Theorem \ref{thm:EKL-65} we obtain that for each $i$-restriction of $f$ the fourth power of its $4$-norm is at most $ q^{144 d^2} \epsilon^2$, where we upper bounded the square of the 2-norm of the $i$-restriction by $\epsilon$.
This shows that $f^2$ is $(2d, q^{144 d^2}\epsilon^2$)-global.
\end{proof}

\begin{proof}[Proof of Theorem~\ref{thm:hypercontractivity}]
Recall that the theorem states that for every $\ell$ a power of 2 and a $(d,\epsilon)$-global function $f$ of degree $d$ we have \[\|f\|_{\ell}^{\ell} \le q^{200d^2\ell^2}\|f\|_2^2 \epsilon^{\ell/2 - 1}.\]
We prove the statement by induction on $\log_2(\ell)$. 
For $g = f^2,$ Lemma \ref{lem:global-square} implies that $g$ is $(2d,q^{144d^2}\epsilon^2)$-global. 
By the induction hypothesis 
\[
\|f\|_{\ell}^{\ell} =  \|g\|_{\ell/2}^{\ell/2} \le q^{50 d^2 \ell^2}\|g\|_2^2(q^{144 d^2}\epsilon^2)^{\ell/4 - 1} \le q^{86 d^2 \ell^2}\|g\|_2^2\epsilon^{\ell/2-2} \le q^{86 d^2 \ell^2}\|f\|_4^4\epsilon^{\ell/2 - 2}.
\]   
We now apply Theorem \ref{thm:EKL-65} to have 
\[
\|f\|_4^4\le q^{103d^2}\epsilon \|f\|_2^2, 
\]
which complete the proof. 
\end{proof}

Theorem~\ref{thm:hypercontractivity} yields the following upper bound on the level $d$ weight of general and Boolean functions.

\begin{cor}\label{cor:level-global}
Let $f \in L^2(\l(V,W))$ such that $f^{= d}$ is $(d,\epsilon)$-global.
Let $\ell\ge 2$ be a power of $2$ and let $\ell' = \frac{1}{1-1/\ell}$ be its H\"{o}lder conjugate. 
Then
\[
\| f^{= d} \|_2^2\le q^{300 d^2 \ell}\epsilon^{\frac{\ell - 2}{2\ell - 2}} \|f\|_{\ell'}^{\ell'}.
\]
In particular for $f\colon \l(V,W)\to \{0,1\}$, since $\bE[f] = \|f\|_{\ell'}^{\ell'}$, we get 
\[
\| f^{= d} \|_2^2\le q^{300 d^2 \ell}\epsilon^{\frac{\ell - 2}{2\ell - 2}}\bE[f].
\]
\end{cor}

\begin{proof}
By H\"{o}lder's inequality
\[
\|f^{=d}\|_2^2 = \langle f^{=d} , f\rangle \le \|f^{=d}\|_{\ell} \|f\|_{\ell'}.
\] 
We can now apply Theorem~\ref{thm:hypercontractivity} to obtain that 

\[
\|f^{=d}\|_{\ell} \le q^{200d^2\ell} \epsilon^{1/2-1/\ell}\|f^{=d}\|_2^{2/\ell}.
\]
Combining the inequalities, we obtain 
\[
\|f^{=d}\|_2^2 \le  q^{200d^2\ell} \epsilon^{1/2-1/\ell}\|f^{=d}\|_2^{2/\ell}\|f\|_{\ell'}.
\]
Hence, after Rearranging
\[
\|f^{=d}\|_2^{2/\ell'} \le q^{200 d^2 \ell} \epsilon^{1/2-1/\ell} \|f\|_{\ell'}.
\]
Raising everyhting to the power $\ell'$ we obtain
\[
\|f^{=d}\|_2^2 \le q^{300 d^2 \ell} \epsilon^{\frac{\ell - 2}{2\ell - 2}} \|f\|_{\ell'}^{\ell'},
\]
as $200 \ell' \le 4/3 \cdot 200 \le 300,$ and $\ell' = \frac{\ell}{\ell -1}.$
\end{proof}

\section{Level inequalities} \label{sec:proof-1.14}

In this section we prove Theorem~\ref{thm:level-inequality} and give a slightly more analytic version of it for functions that are not necessarily Boolean (see Theorem \ref{thm:level-inequality-porism} below). 

%The content of Theorem~\ref{thm:level-inequality} is showing that if a Boolean function $f\colon \l(V,W) \to \{0, 1\}$ is global, then $\|f^{=d}\|_2$ is small. 

%\remove{
%\begin{lem}
%Let $U$ be either a subspace of $V$ of dimension 1 or a subspace of $W$ of codimension $1$. Then for any $f$ we have $(f_{U\to T})^{= d-1} = D_{U,T}[f^{=d}] + (f^{=d-1})_{U\to T}.$
%\end{lem}
%\begin{proof}
%Without loss of generality $U\le V$.As restrictions can only decrease the degree, the lemma can be equivalently stated as $(f^{\ge d})_{U\to T})^{=d-1} = D_{U,T}[f^{=d}].$ By linearity of all operators involved we may assume that $f$ is a character $\chi_X$ of degree $\ge d.$ Let $Y=X(W,V/U)$. Then the rank of $Y$ is the same as the rank of $X$ is $U$ if not contained in its image, and otherwise its rank is $\mathrm{rank}(X) -1$. Therefore $\chi_X$ contributes to the $d-1$ degree part of the restriction if and only if its degree is equal to $d$ and $U$ is contained in its image. On the other hand, the derivative of such a character is zero if $U$ is not contained in its image (the Laplacian annihilates it), and otherwise it contributes its own restriction. 
%\end{proof}
%}

\begin{lem}\label{lem:influences-level}
Let $f \in L^2(\l(V,W))$, and assume that $f^{= d}$ has $(d,\beta \|f^{=d}\|_2^2)$-small generalized influences.
Let $\ell\ge 2$ be a power of $2$, $\ell'$ its H\"{o}lder conjugate, and $\beta \ge 1$. 
Then
\[
\| f^{= d} \|_2^2 \le q^{420 d^2 \ell} \beta^{1- 2/\ell} \|f\|_{\ell'}^2.
\]
\end{lem}

\begin{proof}
By H\"{o}lder's inequality we have 
\[
\|f^{=d}\|_2^2=\langle f^{=d}, f\rangle \le \|f^{=d}\|_{\ell} \|f\|_{\ell'}.
\]
By Proposition~\ref{prop:influence-global} the function $f^{=d}$ is $(d, \beta q^{10d^2}\|f^{=d}\|_2^2)$-global. 
By Theorem~\ref{thm:hypercontractivity} we obtain that 
\[
\|f^{=d}\|_{\ell} \le q^{200 d^2 \ell} \left(\beta q^{10d^2}\right)^{1/2 - 1/\ell} \|f^{=d}\|_2.
\]
Hence, 
\[
\|f^{=d}\|_2^2 \le q^{210d^{2}\ell} \|f\|_{\ell'}\beta^{1/2 -1/\ell}\|f^{=d}\|_2.
\]
Rearranging yields
\[
\|f^{=d}\|_2 \le q^{210 d^2 \ell} \beta^{1/2-1/\ell} \|f\|_{\ell'}.
\]
The lemma now follows by squaring. 
\end{proof}

Recall from the introduction that $\zeta$ is a sufficiently small absolute constant. 
\begin{defn} \label{def:global-ell}
For $\ell' \ge 1$, say that $f \in L^2(\l(V,W))$ is $(r,\epsilon,L^{\ell'})$-\emph{global}, if for each $r$-restriction of $f$ we have 
\[
\|f_{(V',U')\to T}\|_{\ell'} \le \epsilon.
\] 
We say that $f$ is $L^{\ell'}$-\emph{global} if it is $(r,q^{\zeta r n}\|f\|_{\ell'})$-global for all $r$.  
\end{defn}

Note that this is slightly inconcistent with definition~\ref{def:global}, as for the case $\ell'=2$ the norm is squared. So a function is $(d,\epsilon)$-global if and only if it is $(d,\epsilon^{1/2}, L^{2})$-global. 
However working with the $\ell'$-power of the norm for $\ell'\neq 2$ is inconvenient. 

\begin{lem}\label{lem:E_u-global}
Let $f \in L^2(\l(V,W))$. 
If $f$ is $(r,\epsilon,L^{\ell'})$-global, then $\mE_U(f)$ is $(r,2\epsilon,L^{\ell'})$-global.
\end{lem}

\begin{proof}
The same proof of Lemma \ref{lem:E_U-global} works for general $\ell'$-norms.
\end{proof}

Recall the operator $\mT_{d,U}(f)$ from Lemma~\ref{lem:EKL-60}. 
Let us show that it preserves globality. 

\begin{lem}\label{lem:Tdu-global}
Let $f \in L^2(\l(V,W))$, and let $U$ and $\mT_{d,U}(f)$ as in Lemma~\ref{lem:EKL-60}. 
If $f$ is $(d,\epsilon, L^{\ell'})$-global, then $\mT_{d,U}(f)$ is $(d, q^{3d}\epsilon, L^{\ell'})$-global.
\end{lem}

\begin{proof}
This follows from the fact that $\mT_{d,U} f = f -q^{d-1}\mE_U(f) - q^{d}\mE_U(f)+q^{2d-1}\mE_U^2(f),$ Lemma \ref{lem:E_u-global} and the triangle inequality. 
\end{proof}
 
\begin{thm} \label{thm:f=d-global}
Let $\ell\ge 2$ be a power of $2$ and $\ell'$ its H\"{o}lder conjugate. 
Let $f \in L^2(\l(V,W))$ be $(d,\epsilon, L^{\ell'})$-global, and set
$\epsilon' = q^{500d^2\ell}\epsilon^2$. 
Then the function $f^{=d}$ has $(d,\epsilon')$-small generalized influences.
\end{thm}

\begin{proof}
The proof is by induction on $d$. The statement is trivial for $d=0$, and by Lemma~\ref{lem:Tdu-global}, the function $\mT_{d,U}(f)$ is $(d,q^{3d}\epsilon, L^{\ell'})$-global. This implies that the function $(\mT_{d,U}(f))_{U\to T}$ is $(d-1, q^{3d}\epsilon, L^{\ell'})$-global. We can now apply the induction hypothesis and Lemma \ref{lem:EKL-60} to obtain that $D_{U,T}[f^{=d}] = \left((T_{d,U}(f))_{U\to T}\right)^{= d - 1}$ has $(d-1, q^{6d}q^{500(d-1)^2\ell}\epsilon^2)$-small generalized influences. Let $\epsilon''= \max(q^{6d+500(d-1)^2\ell}\epsilon^2, \|f^{=d}\|_2^2).$ Then the function $f^{=d}$ has $(d,\epsilon'')$-small generalized influences. Indeed,  $\|f^{=d}\|_2^2 \le \epsilon''$ and all the $i$-derivatives of $f$ are $(i-1)$-derivatives of  $1$-derivatives of $f$ by Proposition \ref{prop:EKL-38}, so the desired upper bound on the  $L^2$-norm of the $i$-derivatives follows from our bound on the small generalized influences of the $1$-derivatives. \gnote{explain that this is since we proved for all $1$ derivatives} Now if $\epsilon'' = q^{6d+500(d-1)^2\ell}\epsilon^2$, then $\epsilon '' \le q^{500 d^2\ell}\epsilon^2$ and then we are done. Otherwise, $\epsilon'' = \|f^{=d}\|_2^2$ we may now apply Lemma \ref{lem:influences-level} with $\beta =1$ to obtain that 
\[
\|f^{=d}\|_2^2 \le q^{420 d^2 \ell}\|f\|_{\ell'}^2.
\]
This completes the proof as the $(d,\epsilon, L^{\ell'})$-globalness of $f$ implies that $\|f\|_{\ell '}\le \epsilon,$ which in turn implies that $\epsilon'' \le q^{420 d^2 \ell}\epsilon^2\le q^{500d^2\ell}\epsilon^2.$  
\end{proof}

\begin{thm} \label{thm:level-inequality-porism}
Let $\ell\ge 2$ be a power of 2 and let $\ell'$ be its H\"{o}lder conjugate. 
Then for all $(d,\epsilon, L^{\ell'})$-global functions $f\in L^2(\l(V,W))$ we have 
\[
\|f^{=d}\|_2^2 \le q^{460d^2\ell} \epsilon^{\frac{\ell-2}{\ell -1}}\|f\|_{\ell'}^{\ell'}.
\] 
\end{thm}

%\begin{thm}
%Let $\ell\ge 4$ be a power of 2. Then for all $(d,\epsilon)$-global functions $f\colon \l(V,W)\to \{0,1\}$ we have 
%\[
%\|f^{=d}\|_2^2 \le q^{460d^2\ell}\bE^{1-1/\ell}[f]\epsilon.
%\] 
%\end{thm}

\begin{proof}
By Theorem~\ref{thm:f=d-global} the function $f^{=d}$ has $q^{500d^2 \ell}\epsilon^2$-small generalized influences. By Lemma \ref{lem:influences-level} we then obtain that 
\[
\|f^{=d}\|_2^2 \le q^{420d^2\ell}\left( \frac{q^{500d^{2}\ell}\epsilon^{2}} {\|f^{=d}\|_2^{2}}\right)^{1-2/\ell}\|f\|_{\ell'}^2. 
\]
Rearranging, we obtain that 
\[
    \|f^{=d}\|_2^{4 - 4/\ell} \le q^{(420 + 500(1 - 2/\ell)) d^{2}\ell}\epsilon^{2 - 4/\ell} \|f\|_{\ell'}^{2}.
\]
Rearranging, and raising everything to the power $\ell'/2 = \frac{1}{2-2/\ell}$ we obtain that 
\[
\|f^{=d}\|_2^{2} \le q^{\frac{420 + 500(1-2/\ell)}{2-2/\ell} d^2 \ell} \epsilon^{\frac{\ell -2}{\ell -1}} \|f\|_{\ell'}^{\ell'}\le q^{460 d^2\ell} \epsilon^{\frac{\ell - 2}{\ell -1}}\|f\|_{\ell'}^{\ell'}.
\]
\end{proof}

\begin{proof}[Proof of Theorem \ref{thm:level-inequality}]
 Suppose that $f\colon \l(V,W)\to \{0,1\}$ is $(d,\epsilon)$-global.  Then it is $(d, \epsilon^{1/\ell'}, L^{\ell'})$-global. The statement now follows from  Theorem \ref{thm:level-inequality-porism} while plugging in $\epsilon^{1/\ell'} = \epsilon ^{\frac{\ell-1}{\ell}}, \|f\|_{\ell'}^{\ell'} = \mathbb{E}[f]$ to obtain that 
 \[
 \|f^{=d}\|_2^2 \le q^{460 d^{2} \ell} \epsilon^{\frac{\ell -2}{\ell}}\mathbb{E}[f]. 
 \]
\end{proof}
The following is an immediate corollary. 

\begin{cor}\label{cor:global-level-small}
Let $f\colon  L^2( \l(V,W))\to \{0,1 \}$ be $(d,\epsilon)$-global.
Let $t>0$ be such that $\epsilon \ge q^{-t^2}$. 
Then 
\[
\|f^{=d}\|_2^2 \le q^{922dt}\epsilon\bE[f],
\]
\end{cor}

\begin{proof}
We may upper bound 
\[
\|f^{=d}\|_2\le \|f^{\le d}\|_2\le \|f\|_2 = \sqrt{\bE[|f|]}.
\]
This shows that the corollary holds trivally when $d>t/10$. Suppose otherwise, 
noting that $t\le n$ we obtain that  $d<n/10$, and let $\ell = 2^{\lceil\log_2(t/d)\rceil}$. 
Note that the proof of Theorem \ref{thm:level-inequality} follows through when the hypothesis that $f$ takes values in $\{0,1\}$
The corollary now follows by plugging in the value of $\ell$ in Theorem~\ref{thm:level-inequality} while noting that $\epsilon^{-2/\ell} \le q^{2dt}$ and $q^{460d^{2}\ell} \le q^{920 d t}$. 
\end{proof}

\section{Levels on $\SLV$, $\GLV$ and $\l(V,V)$} \label{sec:proof-1.12}

In this section we Prove Lemma \ref{lem:level-level}. We then deduce from the lemma hypercontractive and level inequalities in the nonabelian setting from the ones in the abelian setting. 

Throughout this section we set $G$ to be either $\SLnq = \SLV$ or $\GLnq = \GLV$, where $V=\F{q}^n$.
Since $G \subset \l(V,V)$, we have the following two $G\times G$-equivariant linear maps
\[
i \colon L^2(\l(V,V)) \to L^2(G), \qquad i(f) = f|_G, 
\]
\[
j \colon L^2(G) \to L^2(\l(V,V)), \qquad  j(f)(x) = \left\lbrace \begin{array}{cc} f(x) & x\in G \\ 0 &  x\not\in G \end{array} \right..
\]

\begin{defn}[Globalness for functions over $G$] \label{def:global-G}
We say that $f\in L^2(G)$ is $(r,\epsilon)$-global (resp. $(r,\epsilon, L^{\ell'})$-global) if $j(f)$ is $(r,\epsilon)$-global (resp. $(r,\epsilon, L^{\ell'})$-global) as in Definition \ref{def:global} (resp. Definition \ref{def:global-ell}).
\end{defn}

In \cite{gurevich2021harmonic}, Gurevich and Howe introduced the following notions: 
Let $\omega$ be the permutation representation of $G$ on $L^2(\F{q}^n)$, given by $\omega(g)f(x) = f(g^{-1}x)$ for any $f\in L^2(\F{q}^n)$ and $g,x\in G$, and let $\omega^{\otimes d}$ be the $d$-fold tensor product of $\omega$, for any $0\leq d \leq n$.
Let $\rho$ be irreducible representation of $G$.
By \cite[Def.~1.2.2]{gurevich2021harmonic}, we say that $\rho$ is of \emph{strict tensor rank} $d$, if it is a subrepresentation of $\omega^{\otimes d}$, but not of $\omega^{\otimes (d-1)}$.
\cite[Def.~1.2.3 and 3.1.1]{gurevich2021harmonic}, we say that $\rho$ is of \emph{tensor rank} $d$, if it can be written as $\rho' \otimes \chi$ for $\rho'$ of strict tensor rank $d$ and a multiplicative character $\chi$ of $G$, i.e. a complex group homomorphism on $G$. 
By \cite[Rem.~3.1.2]{gurevich2021harmonic}, for $\SLnq$ the notions of strict tensor rank and tensor rank coincide. 

Recall from the introduction that the level of a function $f$ on $\SLV$ is the minimal $d$ for which $f$ is spanned by the matrix coeffecients of representations of tensor rank $\le d$. For $\GLV$ we set the level of a function in a similar fashion. We also set the \emph{strict level} of a function $f$ on $G$ to be the minimal $d$ for which $f$ is a linear combination of matrix coefficients of representations of strict tensor rank $\le d.$ We set $L^2(G)_{\le d}$ be the space of complex valued functions on $G$ of level $\le d$. Finally, we write $L^2(G)_{= d}$ for the space of functions of level $\le d$ on $G$ that are orthogonal to all the functions of degree $\le d-1$ with respect to the $L^2$-inner product given by $\langle f , g \rangle := \mathbb{E}_{A\sim G}[f(A)\overline{g}(A)].$ We write $f_{\le d}$ for the projection of $f$ onto $L^2(G)_{\le d}$ and $f_{=d}$ for the projection of $f$ onto $L^2(G)_{=d}.$

The following proposition is due to Gurevich and Howe. 
\begin{prop}[{\cite[Proposition~9.1.3]{gurevich2020rank}}] \label{prop:GH-9.1.3}
Let $H \le G$ be the subgroup of matrices fixing a subspace of dimension $d$. 
Then an irreducible representation of $G$ is of strict tensor rank $d$ if and only if it contains a non-zero $H$-invariant vector.
\end{prop}

 Functions on $\{-1,1\}^n$ that depend on $d$-coordinates are termed $d$-\emph{juntas} in the Boolean functions jargon.  The following is an analogue notion for functions on $G$.   
\begin{defn}[Juntas over $G$] \label{def:junta}
We say that $f \in L^2(G)$ is a $d$-\emph{junta} if there exist a pair $(U,g)$, of a subspace $U \le V$ of dimension $d$ and a complex valued function $g$ on  $\{ A \in \l(U,V)\,:\, \rank(A)=d \}$, such that $f(A)= g(A|_U)$.
\end{defn}

We now show that $H$-invariance for a function can be equivalently stated as the condition that $f$ is a junta.
\begin{lem}\label{lem:junta-stab}
A function in $L^2(G)$ is a $d$-junta if and only if there exists $U \le V$ of dimension $d$, such that $f$ is invariant under the right action of $H = \{A\in G \,|\, A|_U = I_U\}$, the subgroup of matrices fixing $U$.
\end{lem}

\begin{proof}
If $f$ is a $d$-junta, let $(U,g)$ be as Definition \ref{def:junta}, then for any $h \in H$ we get that for any $A\in G$, 
\[
h.f(A) = f(A h) = g((Ah)|_U) = g(A|_U) = f(A),
\]
therefore $f$ is $H$-invariant.
If $f$ is $H$-invariant, then for $A,B \in G$ with $A|_U = B|_U$, we may choose $h \in H$ with $Ah = B$, namely $h = A^{-1}B$ and therefore $f(A)=f(B)$. 
\end{proof}

Recall that for $\rho$ an irreducible representation of $G$, we denote by $L^2(G)_\rho$ the space of matrix coefficients of $\rho$, which by the Peter-Weyl Theorem is the isotypic component of $\rho$ in $L^2(G)$ and it is isomorphic to $\rho \otimes \rho^*$ as a $G\times G$-representation.

\begin{prop} \label{prop:junta-tensor}
Let $\rho$ be an irreducible representation of $G$ of strict tensor rank $d$.
Then there exists a $d$-junta in the isotypic component of $\rho$.  Moreover, every function in the isotypic component of $\rho$ is a linear combination of $d$-juntas.
\end{prop}
\begin{proof}
First note that since $\omega$ is a real representation it is self dual. Hence, the strict tensor rank of $\rho^*$ is the same as the strict tensor rank of $\rho$. 
By Proposition~\ref{prop:GH-9.1.3}, $\rho^*$ contains a non-zero $H$-invariant vector $v$, where $H \le G$ is the subgroup of matrices fixing a vector space of dimension $d$. This shows that for every $u\in \rho$ the vector $u\otimes v$ is $H$-invariant with respect to the right-action of $G$. By Lemma~\ref{lem:junta-stab} each such function is a $d$-junta.
Since, $\rho^{*}$ is irreducible we get that it is equal to the span of the $G$-translations of $v$, and therefore every function in $\rho \otimes \rho^{*}$ is a linear combination of right-translates of $d$-juntas of the form  $u \otimes v.$
It follows from Lemma~\ref{lem:junta-stab} that  the translation of a $d$-junta is again a $d$-junta, which completes the proof of the proposition.
\end{proof}

Recall that $f^{=d} = j(f)^{=d}$ and $f^{\le d} = j(f)^{\le d}$, for any $d\le n$.
\begin{prop}\label{pro:junta-level}
Let $f \in L^2(G)$ be a $d$-junta.
Then
\[
\frac{|G|}{|\l(V,V)|}\|f\|_2 \le \|f^{\le d}\|_2.
\]  
\end{prop}
\begin{proof}
Since $f$ is a $d$-junta, let $(U,g)$ be as in Definition~\ref{def:junta}.
Define $\tilde{f} \in L^2(\l(V,V))$ by setting $\tilde{f}(A) = g(A|_U)$ if $\rank(A_U)=d$, and $\tilde{f}(A)=0$ otherwise, and note that $\tilde{f}(A) = f(A)$ for any $A\in G$.
Since $j(f)(A)=0$ for $A\not\in G$, and  $j(f)(A) = \tilde{f}(A) = f(A)$ for $A \in G$, we have 
\[
\frac{|\l(V,V)|}{|G|} \langle j(f), \tilde{f} \rangle = \frac{|\l(V,V)|}{|G|} \bE_{A \sim \l(V,V)} j(f)(A) \overline{\tilde{f}(A)}
=  \bE_{A \sim G} j(f)(A) \overline{\tilde{f}(A)} =  \bE_{A \sim G} |f(A)|^2 = \|f\|_2^2.
\]
We also have 
\[
\|\tilde{f}\|_2^2 = \|f\|_2^2 \cdot \Pr_{A\sim \l(V,V)}[\rank(A|_U)=\dim(U)]\le \|f\|_2^2.
\]
Since $\tilde{f}$ is a $d$-junta it is of degree $d$ and we have 
$ \langle j(f), \tilde{f} \rangle = \langle f^{\le d}, \tilde{f} \rangle$. Combining all of the above, together with the Cauchy--Schwarz inequality, we get
\[
\frac{|G|}{|\l(V,V)|}\|f\|_2^2 = \langle j(f), \tilde{f} \rangle = \langle f^{\le d}, \tilde{f}\rangle \le \|f^{\le d}\|_2 \|\tilde{f}\|_2 \le  \|f^{\le d}\|_2 \|f\|_2.
\]
After dividing by $\|f\|_2$ we get the claim.
\end{proof}

We now introduce the operator 
\[
T_d \colon L^2(G) \to L^2(G), \qquad T_d(f) = i(j(f)^{\le d}).
\]
It is easy to observe that the operator $T_d$ is a $G\times G$-morphism. We also note that the adjoint operator $j^*$ is equal to $\frac{|G|}{|L(V,V)|}i$ and therefore the operator $T_{d}$ is self adjoint since it takes the form $\frac{|G|}{|L(V,V)|} j^{*}P_d^* P_d j$, where $P_d\colon L^2(V,W)\to L^2(V,W)$ is the self adjoint operator given by $P_d(f) = f^{\le d}$.
This shows that the isotypic decomposition of $L^2(G)$ into irreducible $G\times G$-modules of the form $\rho\otimes \rho^*$ is a refinement of the spectral decomposition of the self adjoint operator $T_d$. The following lemma is an operator theoretic version of Lemma \ref{lem:level-level}. 

\begin{lem} \label{lem:level-level-2}
Let  $f \in L^2(G)$ be of strict level $d$. 
Then 
\[\|T_df\|_2 \ge \frac{|G|}{|\l(V,V)|} \|f\|_2.
\]
\end{lem}

\begin{proof}
Since $T_{d}$ commutes with the action of $G\times G$ on $L^2(G)$ and each isotypic space $L^2(G)_{=\rho}$ is an irreducible $G\times G$-representation appearing with multiplicity one inside $L^2(G)$, it follows that the eigenspace of $T_{d}$ are $G\times G$-invariant and can be decomposed as direct sums of isotypic components. Moreover, $T_d$ is self adjoint and positive semi-definite (because as mentioned it is a positive multiple of $j^*P_dP_d j$.) 

Therefore, in order to prove the lemma it suffices to show the for each $\rho$ of strict tensor rank $\le d$ there exists some nonzero function $h_{\rho}\in (L^2(G))_{=\rho}$ with $\|T_d h_{\rho }\|_2 \ge \frac{|G|}{|\l(V,V)|} \|h_{\rho}\|_2$. Indeed, the lemma would then follows by decomposing  $f$ orthogonally as a sum of eigenfunctions $f_{\rho} \in L^2(G)_{=\rho}$, where the eigenfunctions $f_{\rho}$ correspond to the eigenvalue $\frac{\|T_{d}h_{\rho}\|_2}{\|h_{\rho}\|_2}.$

By Proposition \ref{prop:junta-tensor} there exists an eigenfunction $h_{\rho}\in L^2(G)_{=\rho}$ that is a $d$-junta.
By Proposition~\ref{pro:junta-level},
\[
\lambda\|h_{\rho}\|_2^2 = \langle T_d h_{\rho} , h_{\rho} \rangle = \langle j(h_{\rho})^{\le d}, i^*h_{\rho} \rangle  = \frac{|\l(V,V)|}{|G|} \|j(h_{\rho})^{\le d}\|_2^2\ge  \frac{|G|}{|\l(V,V)|}\|h_{\rho}\|_2^2.
\]
\end{proof}

\begin{proof}[Proof of Lemma \ref{lem:level-level}]
For every $f$ we have $\langle T_d f, f \rangle = \langle f^{\le d}, i^{*} f \rangle  = \frac{|L(V,V)|}{|G|} \|j(f)^{\le d}\|_2^2.$  The statement now follows from Lemma \ref{lem:level-level-2}.  
\end{proof}

\subsection{Bonami and level inequalities for strict tensor rank}

We may now capitalize on Lemma \ref{lem:level-level-2}, which is an operator theoretic version of Lemma \ref{lem:level-level}, to deduce Bonami-type lemmas and level-$d$ inequalities on $L^2(G)$ from their $\l(V,V)$ variants. We begin with a variant of Bonami's lemma.

\begin{thm}\label{thm:Bonami-type lemma strict}
Let $d\in \mathbb{N}$, $\ell\ge 4$ be a power of $2,$ and $\rho$ be a representation of $G$ of strict tensor rank $d$. If a function $f\in L^2(G)_{=\rho}$ is $(d,\epsilon)$-global, then 
\[
\|f\|_{\ell}^{\ell} \le q^{1212d^2\ell^2}\|f\|_2^2 \epsilon^{\ell/2 - 1}.
\]
\end{thm}

\begin{proof}
Since $T_d$ is a $G\times G$-morphism, every function $f\in L^2(G)_{=\rho}$ is an eigenfunction of $T_d$ and by Lemma \ref{lem:level-level-2} the corresponding eigenvalue is at least $\frac{|G|}{|\l(V,V)|}$ in absolute value. 
We therefore have 
\[
\frac{|G|}{|\l(V,V)|}\|f\|_{\ell}^{\ell} \le \|T_{d} f\|_{\ell}^{\ell} = \|i (f^{\le d})\|_{\ell}^{\ell} = \frac{|\l(V,V)|}{|G|}\| f^{\le d}1_{G} \|_{\ell}^{\ell} \le \frac{|\l(V,V)|}{|G|}\| f^{\le d}\|_{\ell}^{\ell}.
\]
Rearranging yields 
\[
\|f\|_{\ell} \le \left(\frac{\l(V,V)}{|G|}\right)^{2/\ell}\|f^{\le d}\|_{\ell} \le (4q)^{2/\ell}\sum_{i=0}^{d}\|f^{=i}\|_{\ell}. 
\]
By definition $j(f)$ is $(d,\epsilon)$-global or equivalently $(d,\sqrt{\epsilon},L^2)$-global. We may therefore apply Theorem \ref{thm:f=d-global} to obtain that the function $j(f)^{=i}$ has $(i,q^{1000i^2}\epsilon)$-small generalized influences for each $i\le d$. By Proposition \ref{prop:influence-global} each such $f^{=i}$ is $(i, q^{1010 i^2}\epsilon)$-global. Therefore, by Theorem \ref{thm:hypercontractivity} we have 
\[
\|f^{=i}\|_{\ell}^{\ell} \le q^{1210 i^2 \ell^2} \|f^{=i}\|_2^2 \epsilon^{\ell/2 -1} \le  q^{1210 i^2 \ell^2} \|j(f)\|_2^2 \epsilon^{\ell/2 -1} \le q^{1210 i^2 \ell^2} \|f\|_2^2 \epsilon^{\ell/2 -1} .
\]
The statement follows by summing the above upper bounding for $\|f^{=i}\|_{\ell}$ over all $i$. 
\end{proof}

We now move on to deducing our strict level $d$-inequality. 

\begin{thm}\label{thm:level-GLn-porism}
Let $\ell\ge 2$ be a power of $2$ and $\ell'$ its H\"{o}lder conjugate. 
Let $f \in L^2(G)$ be $(d,\epsilon, L^{\ell'})$-global. Let $g$ be the projection of $f$ onto the space of functions of strict level $\le d$. 
Then  
\[
\|g\|_2^2 \le q^{461 d^2 \ell}\|f\|_{\ell'}^{\ell'}\epsilon.
\]   
\end{thm}

\begin{proof}
We have
\[
\|T_d(f)\|_2^2 = \|i(j(f)^{\le d})\|_2^2 \le  \frac{|G|}{|L(V,V)|} \|j(f)^{\le d}\|_2^2.
\] 
As $j(f)$ is $(d,\epsilon,L^{\ell'})$-global, we may apply Theorem \ref{thm:level-inequality-porism} to obtain that 
\[
\|j(f)^{\le d}\|_2^2 = \sum_{k=0}^{d} \|j(f)^{=k}\|_2^2 \le \|j(f)\|_{\ell'}^{\ell'} \epsilon^{\frac{\ell -2}{\ell -1}} \cdot \sum_{k=0}^{d}q^{460 d^{2}\ell} \le  q^{461d^2 \ell} \frac{|G|}{|L(V,V)|} \|f\|_{\ell'}^{\ell'} \epsilon^{\frac{\ell -2}{\ell -1}}.
\] 
Rearranging, yields that 
\[
\|T_d(f)\|_2^2 \le \left(\frac{|G|}{|L(V,V)|}\right)^2 q^{461 d^2 \ell} \|f\|_{\ell'}^{\ell'} \epsilon^{\frac{\ell -2}{\ell -1}}.
\]

Since $T_d$ is a $G\times G$-morphism it follows that each space $L^2(G)_{=\rho}$ is $T_d$-invariant and therefore $T_d g$ and $T_d(f-g)$ are orthogonal, being Fourier supported on a disjoint set of isotypic components. Lemma \ref{lem:level-level-2} now implies that 
\[
\|g\|_2^2 \le \left(\frac{\l(V,V)}{|G|}\right)^2\|T_d g\|_2^2 \le \left(\frac{\l(V,V)}{|G|}\right)^2 \|T_d f\|_2^2 \le q^{461 d^2 \ell} \|f\|_{\ell'}^{\ell'} \epsilon^{\frac{\ell -2}{\ell -1}}. 
\]
\end{proof}

\subsection{Bonami and level inequalities for tensor rank}
For a function $f\in L^2(G)$ we denote by $f_{\le d}$ the projection of $f$ onto the space $L^2(G)_{\le d}$ of functions of level $d$. We also write $f_{=d} = f_{\le d} - f_{\le d-1}.$ We start by quickly deducing our Bonami Lemma with tensor rank replacing strict tensor rank.

\begin{thm}\label{thm:Bonami-type lemma}
Let $d\in \mathbb{N}$, $\ell\ge 4$ be a power of $2,$ and $\rho$ be a representation of $G$ of tensor rank $d$. If a function $f\in L^2(G)_{=\rho}$ is $(d,\epsilon)$-global, then 
\[
\|f\|_{\ell}^{\ell} \le q^{1212d^2\ell^2}\|f\|_2^2 \epsilon^{\ell/2 - 1}.
\]
\end{thm}
\begin{proof}
Each such function $f$ can be multiplied by a multiplicative character $\chi,$ such that $\chi \cdot f$ is of strict level $d$. Since multiplicative characters have absolute value always equal to 1, the statement follows immediately from Theorem \ref{thm:Bonami-type lemma strict}.
\end{proof}

\begin{thm}\label{thm:level-GLn}
Let $\ell\ge 2$ be a power of $2$ and $\ell'$ its H\"{o}lder conjugate. 
Let $f \in L^2(G)$ be $(d,\epsilon, L^{\ell'})$-global. 
Then  
\[
\|f_{\le d}\|_2^2 \le q^{462 d^2 \ell}\|f\|_{\ell'}^{\ell'}\epsilon.
\]   
\end{thm}

\begin{proof}
For each multiplicative character $\chi$, let $g_{\chi}$ be the projection of $f\cdot \chi$ onto the functions of strict level $d$. Then we may upper bound $\|f_{\le d}\|_2^2 \le \sum_{\chi}\|g_{\chi}\|_2^2.$ Since multiplicative characters have absolute value 1, they do not affect globalness or norms and therefore the statement follows from Theorem \ref{thm:level-GLn-porism} in conjunction with the fact that there are at most $q-1$ multiplicative characters in $G$.
\end{proof}

We may now plug in an optimized value of $\ell$ to obtain the following. 

\begin{thm}\label{thm:flexible-level-GLn}
Let $f\colon G\to \{0,1\}$ be $(d,\epsilon)$-global with $\epsilon \ge q^{-t^2}$, for $t>0$.
Then  
\[
\|f_{\le d}\|_2^2 \le q^{926 d t}\bE[f]\epsilon.
\]
\end{thm}

\begin{proof}
We may assume that $d\le t/10$ for otherwise the statement holds trivally as $\|f_{\le d}\|_2\le \|f\|_2 = \sqrt{E}[f].$ 
Now the function $f$ is $(d,\epsilon^{1/\ell'} ,L^{\ell'})$-global for every $\ell \ge 1$. We apply Theorem \ref{thm:level-GLn} with $\ell = 2^{\lceil \log_2(t/d) \rceil}$ to obtain the statement, while noting that $\epsilon^{\frac{\ell -2}{\ell'(\ell - 1)}} =\epsilon \cdot \epsilon^{-2/\ell} \le q^{2dt}\epsilon.$ 
\end{proof}

We will also make use of the following level $d$-inequalities for global functions.

\begin{thm} \label{thm:simplified-level-GLn}
Let $f\colon G \to \{0,1\}$ be global with $\bE[f] \ge q^{-t^2}$, for some $t > 0$.
Let $\delta> 926\frac{t}{n} + \zeta$.  
Then 
\[
\|f_{=d}\|_2^2 \le q^{\delta d n}\bE[f]^2.
\]
\end{thm}
\begin{proof}
Follows immediately from Theorem~\ref{thm:flexible-level-GLn} with $\epsilon  = q^{\zeta d n}\mathbb{E}[f].$
\end{proof}

\section{Spectral decomposition of global functions} \label{sec:proof-1.15}

Let $G$ be either $\SLnq$ or $\GLnq$.
In this section, we decompose the space $V = L^2(G)$ as an orthogonal direct sum of the $G$-invariant subspaces $V_{=d}$, using the tensor rank notion of Gurevich and Howe \cite{gurevich2021harmonic}.
We give upper bound $\|f*g\|_2$ whenever $g\in V_{=d}$ and $f$ is global, and prove Theorem~\ref{thm:convolution-degree}.

We recall Definition~\ref{def:level-L^2(SL(V))} from the introduction and extend it also to $\GLnq$.
Let $V = L^2(G)$, and by the Peter-Weyl Theorem, $V = \bigoplus_{\rho} L^2(G)_{\rho}$, where $\rho$ runs over all irreducible representations of $G$.
For any $d\le n$, define $V_{=d} = \bigoplus_{\rho} L^2(G)_{\rho}$, where $\rho$ runs over the irreducible representations of $G$ of tensor rank $d$.
Then $V = \bigoplus_{d=0}^n V_d$ is a descomposition of subrepresentations of $G\times G$, where no two summands contain a common factor, hence the decomoposition is orthogonal, and is preserved by convolution from either side.
Denote $V_{\le d} = \bigoplus_{i\le d} V_{=i}$, and $V_{<d}$ and $V_{> d}$ are defined similarly.
For $f\in V$, define by $f_{=i}$ and $f_{\le d}$ the projections of $f$ onto $V_{=i}$ and $V_{\le d}$, respectively.

The following bound was proved in \cite[Theorems~1.2 and 1.3]{guralnick2020character} (see also \cite[Theorem~2.2.1 and Corollary 3.2.7]{gurevich2021harmonic}). 

\begin{thm} \label{thm:GLT}
There exists an absolute constant $c'>0$, such that for any $q$, $n$ and $d\le n$, every irreducible representation of $G$ of tensor rank $d$ has dimension at least $q^{c'dn}$.
\end{thm}

We also make use of the following lemma, which relies on a crucial idea that first appeared in the work of Sarnak and Xue \cite{sarnak1991bounds}. 
They were interested in the operator norm of a self adjoint $G$-endomorphism $T\colon V\to V$, where $V$ is a unitary representation of $G$. 
They then used representation theory to upper bound its operator norm, which is the same as its maximal eigenvalue. 
They first noted that the operator norm is at most the trace of $T$ divided by the multiplicity of the largest eigenvalue of $T$. 
They then used the fact that each eigenspace of $T$ is a subrepresentation of $V$ to deduce that the multiplicity of each eigenvalue of $T$ is at least $m_V$, where $m_V$ is the minimal dimension of an irreducible subrepresentation of $V$. 
This allowed them to deduce
\begin{equation}\label{eq:Sarnak-Xue}
\|T\|\le \frac{\mathrm{tr}(T)}{m_V}.
\end{equation}

For $U\le L^2(G)$ a linear subspace, and $T\colon L^2(G)\to L^2(G)$ a linear operator, denote by $\|T\|_U$ the supremum of $\frac{\|Tf\|_2}{\|f\|_2}$ over all nonzero $f\in U$. 
For $f\in L^2(G)$, define the operator $T_f \colon L^2(G)\to L^2(G)$, by $T_f(g) = f*g$.
Let $m_d$ (resp. $m_{>d}$) denote the minimal dimension of a representation of tensor rank $d$ (resp. $>d$).   

\begin{lem} \label{lem:SX}
For any $d\le n$ and $f\in L^2(G)$, then $V_{=d}$ is $T_f$-invariant and 
\[
\|T_f\|_{V_{=d}}\le \frac{\|f_{=d}\|_2}{\sqrt{m_d}}.
\]
\end{lem}

\begin{proof}
The subspace $V_{=d} \le L^2(G)$ is a subrepresentations of $G \times G$, and therefore is invariant under convolution from both sides, hence in particular $T_f$-invariant. 
This shows that if $f\in V_{=d}$ and $g\in V_{=d'}$ for $d\ne d'$, then $f*g\in V_{=d}\cap V_{=d'} = \{0\}$. 
Hence, $T_f$ and $T_{f_{=d}}$ agree on $V_{=d}$, and by a similar argument, $T_{f}^*$ agrees with $T_{f_{=d}}^*$ on $V_{=d}$. 
Thus by \eqref{eq:Sarnak-Xue} we have 
\[
\|T_f\|_{V_{=d}} = \|T_{f_{=d}}\| = \sqrt{\|{T_{f_{=d}}}^*T_{f_{=d}}\|} \le \frac{\sqrt{tr (T_{f_{=d}}^*T_{f_{=d}})}}{\sqrt{m_d}} = \frac{\left\|f_{=d}\right\|_2}{\sqrt {m_d}},
\]
where the last equality follows from the following well known claim, which we give for completeness. 
\end{proof}    

\begin{claim}
Let $G$ be a finite group, let $f\in L^2(G)$ and let $T \colon L^2(G) \to L^2(G)$, $T(g) = f*g$. 
Then 
\[
\mathrm{tr}(T^*T)=\|f\|_2^2.
\]
\end{claim}

\begin{proof}
For $x\in G$ let $1_x$ be the indicator of $x$, and let $\mu_x \eqdef |G|\cdot 1_x$ and $e_x \eqdef \sqrt{|G|}1_x$. 
Then the functions $e_x$ constitute an orthonormal basis for $L^2(G)$. 
As the convolution with $\mu_x$ is simply a translation by $x$, it preserves $2$-norms, i.e. $\|f*\mu_x\| = \|f\|$.
We therefore get
\[
\mathrm{tr}(T^*T) = \sum_{x\in G} \langle T^*T e_x,e_x\rangle = \sum_{x\in G} \| Te_x\|_2^2 = \sum_{x\in G} |G|^{-1}\|f*\mu_x\|_2^2 = \|f\|_2^2.
\] 
\end{proof}

We obtain the following version of Theorem~\ref{thm:level-GLn} for $\GLnq$.

Recall that $f$ is global if it is $(d,q^{\zeta dn}\mathbb{E}[f])$-global for all $d$.
When using globalness below, we may acquire constraints on $\zeta$ forcing it to be smaller than some other constants. The value of $\zeta$ will be set to the highest constant that satisfies all of these constraints.

\begin{thm}\label{thm:SX-global}  
There exists $c>0$, such that for any $n \in \N$ and any prime power $q$, the following holds. 
Let $1<t<cn$, and let $f\colon G\to \{0,1\}$ be a global function such that $\bE[f] \ge  q^{-t^2}$. 
Then for all $d \ge 1$,
\[ 
\|T_f\|_{V_{=d}} \le q^{-cdn} \cdot \bE[f].
\] 
\end{thm}

\begin{proof}
Let $c'$ denote the constant from Theorem~\ref{thm:GLT}, and assume $c$ and $\zeta$ are sufficiently small such that
\[
(c' - 2c - \zeta) \cdot dn \ge t^2 + 926dt.
\]
Combining  Lemma~\ref{lem:SX} with Theorems~\ref{thm:GLT} and \ref{thm:flexible-level-GLn}, we have
\[
\|T_f\|_{W_{=d}}^2 \le \frac{\left\|f_{=d}\right\|_2^2}{m_d} \le \frac{q^{926dt}\cdot \bE[f]^{2}\cdot q^{\zeta dn} }{q^{c'dn}} \le q^{t^2 + 926dt + \zeta dn - c'dn} \cdot \bE[f]^2 \le q^{-2cdn} \cdot \bE[f]^2.
\]
\end{proof}

We are now in a position to prove Theorem~\ref{thm:convolution-degree}.

\begin{proof}[Proof of Theorem~\ref{thm:convolution-degree}]
Follows from immediately Theorem~\ref{thm:SX-global} applied to $1_A$.
\end{proof}

\section{Mixing and product mixing} \label{sec:proof-1.3,6,7,8}

In this section we prove Theorems~\ref{thm:SLn-mixing-global} and \ref{thm:SLn-product-mixing}, and Corollaries~\ref{cor:nikolov-pyber} and \ref{cor:SLn-stability-Gowers}, as well as giving a extending a result of Keevash and Lifshitz \cite{keevash2023sharp} regarding a mixing property to $\SLnq$ (Theorem~\ref{thm:SLn-global-mixing}).

We begin by proving generalizations of Theorems~\ref{thm:SLn-mixing-global} and \ref{thm:SLn-product-mixing}, applied for either $G=\SLnq$ or $\GLnq$.

\begin{thm}\label{thm:mixing-global}
There exists an absolute constant $c>0,$ such that the following holds.
Let $A,B \subseteq G$ be global sets of density $\mu(A), \mu(B) \ge q^{-cn^2}$, and let $f = 1_A, g = 1_B \colon G \to \{0,1\}$.
Then 
\[
\|f * g - f_{= 0}* g_{= 0}\|_2 \le q^{-n/4} \bE[f]\bE[g].
\]
\end{thm}

\begin{proof}
By the $T_f$-invariance and orthogonality of the decomposition of $V = \bigoplus_{d=0}^n V_{=d}$, we get
\[
\| f*g - f_{=0}*g_{=0}\|_2 = \| T_f(g) - T_f(g_{=0})\|_2 = \sum_{d=1}^n \| T_f(g_{=d})\|_2.
\]
By Theorems~\ref{thm:SX-global} and \ref{thm:simplified-level-GLn}, we have 
\[
\| T_f(g_{=d})\|_2 \le \|T_f\|_{V_{=d}} \|g_{= d}\| \le \frac{q^{\delta d n}}{q^{c'dn}} \cdot \bE[f]\bE[g] \le q^{-n/4} \cdot \bE[f]\bE[g],
\]
where $\delta = 501\frac{t}{n} + \zeta$, $t$ is such that $q^{-t^2} = |G|^{-c}$, $c'$ is the absolute constant of Theorem~\ref{thm:simplified-level-GLn}, and the last inequality holds, provided that $c,\zeta$ are sufficiently small with respect to $c'$. 
The claim follows.
\end{proof}

\begin{proof}[Proof of Theorem~\ref{thm:SLn-mixing-global}]
Follows as a special case of Theorem~\ref{thm:mixing-global} for $\SLnq$.
\end{proof}

\begin{thm}\label{thm:product-mixing} 
There exists an absolute constant $c>0,$ such that the following holds. 
Let $A,B,C \subseteq G$ be global sets of density $\mu(A), \mu(B), \mu(C) \ge q^{-cn^2}$, and let $f = 1_A, g = 1_B, h = 1_C \colon G \to \{0,1\}$. 
Then
\[
|\langle f*g,h \rangle - \langle f_{=0}*g_{=0},h_{= 0}\rangle| \le q^{-n/5} \cdot \bE[f]\bE[g]\bE[h].
\] 
\end{thm}

\begin{proof}
The proof is analogous to that of Theorem~\ref{thm:mixing-global}.
By the $T_f$-invariance and orthogonality of the decomposition of $V = \bigoplus_{d=0}^n V_{=d}$, we get
\[
\langle f*g,h\rangle =\sum_{d=0}^n\langle T_f g_{= d},h_{= d}\rangle,
\]
and by Theorems~\ref{thm:SX-global} and \ref{thm:simplified-level-GLn}, we get 
\[
|\langle T_f g_{= d},  h_{= d}\rangle | \le \|T_f\|_{V_{=d}} \|g_{= d}\| \|h_{= d}\| \le \frac{q^{\delta d n}}{q^{c'dn}} \cdot \bE[f]\bE[g]\bE[h] \le q^{-n/5} \cdot \bE[f]\bE[g]\bE[h].
\]
The claim follows.
\end{proof}

\begin{proof}[Proof of Theorem~\ref{thm:SLn-product-mixing}]
Follows as a special case of Theorem~\ref{thm:product-mixing} for $\SLnq$.
\end{proof}

We are now able to prove Corollaries~\ref{cor:nikolov-pyber} and \ref{cor:SLn-stability-Gowers}.

\begin{proof}[Proof of Corollary~\ref{cor:nikolov-pyber}]
Let $A,B,C \subseteq \SLnq$ be global sets, suppose on the contrary that $ABC \ne \SLnq$, and let $x\notin ABC$ and $f = 1_A, g=1_B, h = 1_{xC^{-1}}$. 
Then $\langle f*g, h \rangle = 0$, which would contradict Theorem~\ref{thm:SLn-product-mixing}
\end{proof}

\begin{proof}[Proof of Corollary~\ref{cor:SLn-stability-Gowers}]
Suppose otherwise that $A$ is a global product free set and let $f=1_A$. 
Then $\langle f*f,f\rangle = 0$, which contracdicts Theorem~\ref{thm:SLn-product-mixing}.
\end{proof}

As usual, we also present the following adaptationt of our results, Theorem~\ref{thm:SLn-product-mixing}, to the non-Boolean setting. 
The analogue of Theorems~\ref{thm:SLn-global-mixing} and \ref{thm:SLn-product-mixing}, takes the following forms. 

\begin{thm}
There exists an absolute constant $c>0,$ such that the following holds. 
Let $\ell$ be a power of 2 and $\ell'$ its H\"{o}lder conjugate. 
Let $f,g \in L^2(\SLnq)$ be $L_{\ell'}$-global functions.  
Then 
\[
\| f * g -\bE[f]\bE[g]\|_2 \le 0.01\|f\|_1\|g\|_1.
\]
\end{thm}

\begin{thm}
There exists an absolute constant $c>0,$ such that the following holds. 
Let $\ell$ be a power of $2$ and let $\ell'$ be its H\"{o}lder conjugate. 
Let $f,g,h\colon \GLnq\to \C$ be $L_{\ell'}$ global functions. 
Then
\[
\left|\langle f*g,h \rangle - \langle f_{= 0}*g_{= 0},h_{= 0}\rangle\right|\le \frac{q^{-n/4}}{q}  \|f\|_1\|g\|_1\|h\|_1.
\] 
\end{thm}

%\begin{thm}
%There exists an absolute constant $c>0,$ such that the following holds. Let $\ell$ be a power of $2$ and $\ell'$ its H\"{o}lder conjugate. Let $f,g,h \in L^2(\SLnq)$ be $L_{\ell'}$-global functions. Then
%\[\left|\langle f*g,h \rangle - \bE[f]\bE[g]\bE[h]\right| \le 0.1  \|f\|_1\|g\|_1\|h\|_1.\] 
%\end{thm}

We now turn to the mixing property which was studied in the work of Keevash and Lifshitz \cite{keevash2023sharp}.

Let $N \in \N$, $[N] = \{1,2,\ldots,N\}$ and $S_N$ the group of permutations on $[N]$.
Let $d\le N$, and for two $d$-tuples of distinct coordinates of $[N]$, $I=(i_1,\ldots, i_d)$ and $ J=(j_1,\ldots, j_d)$, denote by $U_{I\to J} \subset S_N$, the set of permutations satisfying $\sigma(i_r) = j_r$, for any $r=1,\ldots,d$.
Call $U_{I\to J}$ a $d$-umvirate (the case where $d=1$ is called a dictatorship).
\gnote{this should be changed to also allow the adjoint restrictions}

\begin{defn}
Let $\varphi\colon G \hookrightarrow S_N$ be a faithful permutation representation. 
Say that $A\subseteq G$ is $r$-global (w.r.t. $\varphi$) if for each $d \le N$ and each $d$-umvirate $U \subset S_N$ with $\varphi^{-1}(U)\ne \{1\}$, we have 
\[
\frac{|A\cap \varphi^{-1}(U)|}{|\varphi^{-1}(U)|}\le r^d\frac{|A|}{|G|}.
\]
Say that $(G,\varphi)$ is $(r,\epsilon)$-\emph{globally mixing} if for any $A,B,C\subseteq G$ which are $r$-global with $\mu(A), \mu(B), \mu(C) \ge \epsilon$, then 
\begin{equation}\label{eq:mixing}
\frac{|G|^3\Pr_{a,b\sim G}[a\in A,b\in B,ab\in C]}{|A||B|C|}\in (0.99,1.01).
\end{equation}
Let $G_n$ be a sequence of groups, $\varphi_n\colon G_n\hookrightarrow S_{N_n}$ a sequence of permutation representations and $\epsilon_n >0$ a sequence of numbers. 
Denote $\alpha_n$ the minimum of $\frac{|\varphi_n^{-1}(U_{i\to j})|}{|G_n|}$ over dictatorships $U_{i\to j}$  with $\varphi_n^{-1}(U_{i\to j}) \ne \{1\}$. 
Say that the sequence $(G_n,\varphi_n)$ satisfy the $\epsilon_n$-\emph{global mixing property} if there exists $c>0$, such that $(G_n,\varphi_n)$ is $(\alpha_n^{-c},\epsilon_n)$-globally mixing for any $n$.
\end{defn}

Let $\psi_n, \phi_n \colon \SLnq \hookrightarrow S_{q^n}$ be the permutation representations corresponding to the standard and dual actions of $\SLnq$ on $\F{q}^n$, namely $\psi_n(A)(v)=Av$ and $\psi_n(A)(v)=(A^t)^{-1}v$, and let $\varphi_n \colon \SLnq \to S_{2q^n}$ be obtained by concatenating the two actions, i.e. acting via $A$ on the first $q^n$ elements and dually on the last $q^n$ elements.  

\begin{thm}\label{thm:SLn-global-mixing}
There exists an absolute constant $c>0$, such that the following holds.
Let $G_n = \SLnq$ and $\varphi_n\colon G_n \to S_{2q^n}$ as defined above. 
Then the sequence $(G_n,\varphi_n)$ satisfy the $|G_n|^{-c}$-global mixing property.
\end{thm}

In \cite[Thm. 1.12]{keevash2023sharp}, Keevash and Lifshitz showed that $A_n$ satisfies the $e^{-n^{1-c}}$-global mixing property for every $c>0$, where the implicit permutation representations correspond to embedding $A_n$ inside $S_n$. 
In fact, we conjecture that for every sequence of finite simple group of Lie type $G_n$ there exists a sequence of permutation representations $\varphi_n$ and an absolute constant $c>0$, such that $\varphi_n$ satisfy the $|G_n|^{-c}$-global mixing property. 

\begin{proof}[Proof of Theorem~\ref{thm:SLn-global-mixing}]
Let $c$ be half the constant of Theorem~\ref{thm:product-mixing}. 
Let  $A,B,C\subseteq \SLnq$ be global sets of density at least $|G|^{-c} \ge q^{-cn^2}$, and set $f = 1_A, g = 1_B, h = 1_C$. 
Note that 
\[
\langle f_{= 0}*g_{= 0},h_{= 0}\rangle = \bE[f]\bE[g]\bE[h] = \frac{|A||B||C|}{|G|^3},
\]
and
\[
\langle f*g,h \rangle = \Pr_{a,b\sim G}[a\in A,b\in B,ab\in C].
\]
The claim now follows from  Theorem~\ref{thm:product-mixing} which gives 
\[
|\langle f*g, h\rangle -\langle f_{=0}*g_{=0},h_{= 0}\rangle| < q^{-n/4} \cdot \bE[f]\bE[g]\bE[h].
\]
\end{proof}

\section{Polynomial Bogolyubov and approximate subgroups} \label{sec:proof-1.2,9}

%[SE: The following already appeared in the introduction] The Bogolyubov--Rusza Theorem~states that if $A\subseteq \F{p}^n$ has density $\alpha$, then $2A-2A$ contains a subspace of codimension $c_{p}(\alpha)$. Sanders proved that $c_{p,\alpha}$ can be taken to be $O_p(\log^4 (1/\alpha))$ -- this is known as a quasi-polynomial Bogolyubov--Ruzsa result since the density of the guaranteed subspace is $2^{O_p\log^4(1/\alpha))}$. It is a major open problem in additive combinatorics to get rid of the $4$ in the exponent of the $\log$. Such an improvement would lead to polynomial bounds in the inverse Gowers problem and to polynomial Freiman--Rusza. 

In this section we prove the polynomial variant of Bogolubov's lemma for $\SLnq$ (Theorem~\ref{thm:Bogolyubov}). 
We then give an application to the theory of aproximate subgroups (Theorem~\ref{thm:approximate}).

\subsection{Density bumps}

In this section we show that if a set $A$ has a density bump inside an arbitrary $t$-umvirate, then is has a similar density bump inside a good $(2s)$-umvirate, where $s\le 2t$. 

For $v\in V$ and $w\in W$, we write $U_{v\to w}$ for the set of matrices in $\l(V,W)$ sending $v$ to $w$, and for sets of linearly independent vectos $\bar{v}=\set{v_i}_{i=1}^t \subset V$ and $\bar{w}=\set{w_i}_{i=1}^t \subset W$, we write $U_{\bar v\to \bar w} \eqdef \bigcap_{i=1}^d U_{v_i\to u_i}$. 
Similarly, for a pair of vectors $\varphi \in W^*$ and $\psi\in V^*$, we write $U_{(\varphi,\psi)}$ for the set of matrices $A\in \l(V,W)$ such that $A^*\varphi = \psi$, and for linearly independent sets $\bar{\varphi}=\set{\varphi_i}_{i=1}^t \subset W^*$ and $\bar{\psi}=\set{\psi_i}_{i=1}^t \subset V^*$, we write $U_{\bar \varphi \to \bar \psi} \eqdef \bigcap_{i=1}^d U_{\varphi_i\to \psi_i}$. 

As we are interested in the group $\SLnq \cong \SLV \subset \l(V,V)$, we only consider the case where $W=V$, 

\begin{lem}\label{lemma:selecting basis}
Let $(\bar{v},\bar{w})$ and $(\bar \varphi,\bar \psi)$ be as above. 
Then there exists a choice of bases $B =(b_1,\ldots, b_n)$ for $V$ and $\Xi = (\xi_1,\ldots, \xi_n)$ of $V^*$, as well as sub-matrices $M\in Mat_{|\bar V|\times |\bar \varphi|}(\bF_q)$, $P$, and $N$, such that the umvirate $U_{\bar \varphi\to \bar \psi}\cap U_{\bar v, \bar w}$ is represented, with respect to these bases, in the form
\[
\begin{pmatrix}
M &P\\ 
N & X\\
\end{pmatrix},
\]
where $X$ varies over all matrices of appropriate dimensions.
\end{lem}

\begin{proof}
Complete the $v$'s and the $\varphi$'s to full bases. 
\end{proof}

Let $O_{i,j}$ denote the zero matrix with $i$ rows and $j$-columns, and $I_j$ the identity matrix with $j$ rows and columns.

\begin{lem}\label{lemma:iterative basis change}
Consider an umvirate of the form $\begin{pmatrix} M &P\\  N & X\\ \end{pmatrix}$ as in Lemma \ref{lemma:selecting basis} as above, where $M\in Mat_{k,\ell}(\F{q})$. 
Then either $M=0$, or there exists a basis with resepect to which the umvirate is of the form 
\[
\left\{\begin{pmatrix} 
1 & O_{1,\ell-1} & O_{1,n-\ell} \\
O_{k-1,1} &M' & P'\\
O_{n-k,1} &N' &X\\
\end{pmatrix}\right\}_{X}.
\]
Moreover, we choose the bases such that $rank(M')=rank(M)-1$.
\end{lem}

Lemma~\ref{lemma:iterative basis change} easily implies the following in the case where $M$ is invertible.

\begin{cor}\label{cor:upper-left I}
Consider an unverate of the form $\begin{pmatrix}
M &P\\ 
N & X\\
\end{pmatrix}$ 
as in Lemma \ref{lemma:selecting basis} as above, where $M\in Mat_{k,k}(\F{q})$ is an invertible matrix. 
Then there exists a basis with resepect to which the umvirate is of the form 
\[
\left\{\begin{pmatrix} 
I_k &O_{k, n-k}\\
O_{n-k, k} &X\\
\end{pmatrix}\right\}_{X},
\]
and is a good $2k$-umvirate.
\end{cor}

\begin{proof}
This follows by iteratively applying Lemma~\ref{lemma:iterative basis change}, noting that each $M'$ which is formed in the process is invertible. 
Also, note that a basis change is in fact equivalent to multiplication by an invertible matrix, either from the left or from the right.
\end{proof}

\begin{lem}\label{lem:good-umvirates}
For any $t$-umvirate $U$ there exists an $s\le 2t$ such that $U$ can be partitioned into a disjoint union of good $s$-umvirates.  
\end{lem}

\begin{proof}
Iteratively using Lemma~\ref{lemma:iterative basis change} we may assume without loss of generality that $U$ is of the form
\[
\left\{\begin{pmatrix}
I_h & O_{h, \ell-h} & O_{h,n-\ell} \\
O_{k-h, h} &O_{k-h, \ell-h} &P'\\
O_{ n-k, h} &N' &X\\
\end{pmatrix}\right\}_X.
\]
We can then find a minors in both $P'$ and $N'$ that are of full rank - otherwise $U$ does not contained invertible matrices. 
Without loss of generality, assume that $P'=(P'', P''')$, where $P''\in Mat_{k-h,k-h}(\F{q})$ is invertible. 
We can also define $N''$ and $N'''$ similarly in $N'$, where $N''$ has $\ell-h$ rows and columns. 
Now, consider any fixing of the left-top $\ell-h, k-h$ of the $X$ sub-matrix. 
This corresponds to an umvirate of size $t+k-h+\ell-h = k+\ell + k-h+\ell-h \le 2k+2\ell = 2t$. 
Moreover, since the upper-left $k+\ell-h$ by $k+\ell-h$ minor of the matrices is now fixed and invertible, we are done by Corollary~\ref{cor:upper-left I}.
\end{proof}

\begin{lem}\label{lem:good-bump}
Suppose that $A\subseteq \SLnq$ is not $r$-global. 
Then there exists a $t>0$ and a good $t$-umvirate of $A$ in which the density of $A$ is $\ge$ $r^{t/2}\mu(A)$.
\end{lem}

\begin{proof}
If $A$ is not $r$-global, then by definition, there exists an $s$-umvirate $U'$, for some $s$, where the density of $A$ is at least  $r^{s}\mu(A)$. 
By an averaging argument, it follows from Lemma~\ref{lem:good-umvirates} that there exists a good $t$-umvirate $U\subseteq U'$ where the density of $A$ is also bounded below by $r^{s}\mu(A)\ge r^{t/2}\mu(A)$.  
\end{proof}

%\remove{\subsection{Every set contains a global restriction}

%We now show that every set $A \in \l(V,W)$ of density $\alpha$ has a  restriction inside a $t$-umvirate $U_{I\to J}$ that is global for a suitable  dependence of $t$ on $\alpha,r$.

%\begin{lem}\label{lem:the existence of a global restriction}
%Let $\alpha\in\left(0,1\right)$. Suppose that $A\subseteq \l(V,W)$ has density $\alpha \ge q^{-t^2}$. Then there exist an $r$-umvirate $U$ for  $r \le \frac{2 t^2}{\zeta n}$, such that the restriction of $A$ to $\l(V,W)$ is global. 
%\end{lem}

%\begin{proof}
%If $A$ is $r$-biglobal, then we are done. Otherwise we may find a restriction with 
%\[\mu\left(A_{I\to J}\right) \ge r^{\left|I\right|} \mu(A).\]
%We now obtain that the restriction $A_{I\to J}$ is either $r$-global in which case we are done or it has a retriction with a similar density bump. Proceeding like that repeatedly we eventually find an $r$-global restriction. 
%\end{proof}
%}

\subsection{Growth in $\SLnq$}

Above we showed that the density of a non-global set can be increased by considering its restriction inside a good umvirate. 
We can thus keep increasing the density until either it is maximized, or we have a global restriction relative to a good-umvirate.

\begin{defn}[Relative globality]
Let $A\subseteq \SLnq$ be a set and $k \le n$.
Recall that good $k$-umvirate is a set of the form $U=U_k^{g,h} \eqdef g L_k h$, where $g,h\in \SLnq$ and $L_k \le \SLnq$ is isomorphic to $\Sl{n-k}{q}$.
We say that $A$ is global relative to $U$ if $g^{-1}Ah^{-1} \cap L_k$ is global as a subset of $\Sl{n-k}{q}$.
\end{defn}

\begin{lem}\label{lem:global-umvirate}
Let $t>0$ and let $A\subseteq \SLnq$ be a set of density $\ge q^{-t^2}$. 
Then there exists a good $k$-umvirate $U=U^{g,h}_k$, where $k\le \frac{4t^2}{n}$, and such that $A$ is $q^{\zeta n/2}$-global relative to $U$.
\end{lem}

\begin{proof}
We note that good umvirates inside $\Sl{n-k}{q}$ lift to good umvirates of $\SLnq$ when identifying between $\Sl{n-k}{q}$ and $L_k$. 
If $A$ is \emph{not} global (otherwise we are done), we use Lemma~\ref{lem:good-bump} iteratively to increase the density of $A$ inside a good umvirate $U_{k}^{g,h}$ in which $A$ has density $\ge q^{\zeta nk/4}\mu(A)$ until we get stuck, namely $A$ is relatively global. 
As $\mu(A)\ge q^{-t^2}$ we have $k\le \frac{4t^2}{n}$. 
This completes the proof of the lemma.
\end{proof}

\begin{cor}\label{cor:products-large} 
There exists an absolute constant $c>0$, such that the following holds. 
Let $A,B$ be global and suppose that $\mu(A),\mu(B)\ge q^{-cn^2}$. 
Then $\mu(AB) \ge 0.99$. 
\end{cor}

\begin{proof}
Let $f=\frac{1_A}{\mu(A)}$ and $g = \frac{1_B}{\mu(B)}$. 
By Cauchy--Schwarz we have $\|f*g - 1\|_1 \le \|f*g -1\|_2 \le q^{-n/4},$ where we used Theorem~\ref{thm:SLn-mixing-global}. 
Let $\nu$ be the probability distribution obtained by sampling $a\sim A,b\sim B$ and outputting $ab$. 
Then the total variation distance between $\nu$ and the uniform distribution is $\frac{1}{2}\|f*g - 1\|_1\le q^{-n/4}.$ This shows that $AB$, which is the support of $\nu$, has uniform measure $\ge 1 - 0.01 = 0.99$.  
\end{proof}

Recall that a good umvirate is a set of the form $U=U_k^{g,h} = g L_k h \subset \SLnq$, and in the case where $h=g^{-1}$, then $U$ is a good \emph{groumvirate}, as defined in Definition~\ref{def:groumvirate}.

\begin{thm}\label{thm:density-Bogolyubov}
There exists absolute constants $c,C>0$, such that the following holds. 
Let $A\subseteq \SLnq$ be a set of density $\ge q^{-cn^2}$. 
Then there exists $k$ and a good groumvirate $U_{k}^{g,g^{-1}}$ of density $\ge \mu(A)^C$, in which $AA^{-1}$ has density $\ge 0.99$. 
\end{thm}

\begin{proof}
Let $t$ be the smallest such that the density of $A$ is at least $q^{-t^2}$. 
By Lemma~\ref{lem:global-umvirate}, there exists $k< \frac{4t^2}{n}$ and a good umvirate $U^{g,h}_{k'}$ where $A$ is $q^{\zeta n/2}$-global relative to $U^{g,h}_{k}k$ and has density $\ge q^{\zeta nk/4}\cdot \mu(A)\ge \mu(A)$ there. 
Let $U'\eqdef U^{g,h}_k\left(U^{g,h}_k\right)^{-1} = U^{g,g^{-1}}$. 
Then $U'$ is a good $k$-groumvirate, and by Corollary~\ref{cor:products-large} we have that $AA^{-1}\cap U'$ has density $0.99$ in $U'$. 
The density of $U'$ is at least $q^{-2kn}\ge q^{-2t^2}\ge \mu(A)^c$ as desired.
\end{proof}

\begin{thm}[Bogolyubov Ruzsa analogue]\label{thm:Bogolyubov-Ruzsa}
There exists absolute constants $c,C>0$, such that the following holds. 
Let $A\subseteq \SLnq$ be of density $\ge q^{-cn^2}$. 
Then $AA^{-1} A A^{-1}$ contains a good groumvirate of density $\ge \mu(A)^C$.
\end{thm}

\begin{proof}
By Theorem~\ref{thm:density-Bogolyubov} there exists a good groumvirate $U = U_{k}^{g,g^{-1}}$ of density $\ge \mu(A)^C$, in which $AA^{-1}$ has density $\ge 0.99$. 
Assume in contradiction that $x\in U\setminus AA^{-1}AA^{-1}$. 
Then $AA^{-1}\cap U$ and $xA^{-1}A\cap U$ are two disjoint sets of density $0.99$ inside $U$, which is absurd. 
Hence $AA^{-1}AA^{-1}$ contains $U$.   
\end{proof}

\begin{proof}[Proof of Theorem~\ref{thm:Bogolyubov}] 
The Theorem~is an immediate corollary of Theorem~\ref{thm:Bogolyubov-Ruzsa}. 
Indeed, Let $c_1,C_1$ be the constants $c,C$ respectively of Theorem~\ref{thm:Bogolyubov-Ruzsa}. 
Let $C = \max(C_1,\frac{2}{c_1})$. 
Then if $\mu(A)\ge q^{-cn^2}$, then the statement follows from Theorem~\ref{thm:Bogolyubov-Ruzsa} and otherwise it is trivial by taking the subgroup $\{1\}$ as our good groumvirate.  
\end{proof}

\subsection{Approximate groups}

Let $G$ be a group. 
Recall that a set $A\subseteq G$ is said to be a $K$-approximate subgroup if $A =A^{-1}$ and there exists a set $X$ of size $K$, such that $A^2 \subseteq X \cdot A$. 
In this subsection we show that approximate subgroups are contained in the union of a few cosets of a large good umvirate. 
Results of a similar spirit were obtained by Breulard, Green, and Tao \cite{breuillard2011approximate} in the case where $n$ is $O(1)$.   

For $\alpha < \beta$, we say that $A\subseteq \SLnq$ is an $(\alpha,\beta)$-\emph{easy} set if there exists a good groumvirates $U = U_{k}^{g,g^{-1}}$ of density $\ge \alpha$, such that $A$ is a union of $s$ left cosets of $U$ for some $s\le \frac{\beta}{\alpha}$. 

\begin{thm}\label{thm:Approximate-groups}
There exist absolute constants $c,C>0$, such that the following holds. 
Let $A\subseteq \SLnq$ have density $\alpha\ge q^{-cn^2}$ and $A^2A^{-1}AA^{-1}$ have density $\le \beta$. 
Then there exists an $(\alpha^C, \beta)$-easy set $J$, such that $A\subseteq J\subseteq A^5$.
\end{thm}

\begin{proof}
By Theorem~\ref{thm:Bogolyubov-Ruzsa} there exist a good groumvirate $U$ of density $\ge \alpha^C$, such that $U\subseteq A^{-1}AA^{-1}A$. 
Let $X\subseteq A$ be the subset of $A$ obtained by choosing a representative of each left coset of $U$ that $A$ intersects. 
Then $A\subseteq XU \subseteq A^5$. 
Moreover $|XU| = |X||U|$, hence $|X|\le \frac{\beta}{\alpha}$. 
Setting $J= XU$ completes the proof. 
\end{proof}

\begin{proof}[Proof of Theorem~\ref{thm:approximate}]
The Theorem is an immediate corollary of Theorem~\ref{thm:Approximate-groups}.
\end{proof}

\begin{thm}\label{thm:approximate groups2}
There exist absolute constants $c,C>0$, such that the following holds. 
Let $A$ have density $\alpha\ge q^{-cn^2}$. 
Suppose that $A$ is a $K$-approximate subgroup. 
Then $A$ is contained in an $(\alpha^C, K^4\alpha)$-easy set. 
\end{thm}

\begin{proof}
If $A$ is a $K$-aproximate subgroup, then $A=A^{-1}$ and $|A^5|\le K^4|A|$. 
The Theorem now follows from Theorem~\ref{thm:Approximate-groups}.
\end{proof}

\newcommand{\etalchar}[1]{$^{#1}$}

\end{document}